\documentclass[11pt
]{elsarticle}
\usepackage{float}
\usepackage{amssymb}
\usepackage{amsmath}
\usepackage{amscd}
\usepackage{amsfonts}
\usepackage{amsthm}
\usepackage{graphicx}
\usepackage{mathtools}
\usepackage{color}
\usepackage[colorlinks, linkcolor=red, citecolor=green,filecolor=magenta,urlcolor=cyan]{hyperref}

\newfont{\Blackboard}{msbm10 scaled 1200}

\newfont{\roma}{cmr10 scaled 1200}

\newtheorem{theorem}{{}\hskip\parindent Theorem}[section]
\newtheorem{lemma}{{}\hskip\parindent Lemma}[section]

\newtheorem{example}{{}\hskip\parindent Example}[section]
\newtheorem{corollary}{{}\hskip\parindent Corollary}[section]
\newtheorem{definition}{{}\hskip\parindent Definition}[section]
\newtheorem{remark}{{}\hskip\parindent Remark}[section]

\numberwithin{equation}{section}

\def\be{\begin{equation}}
	\def\ee{\end{equation}}
\def\beq{\arraycolsep=1.5pt\begin{eqnarray}}
	\def\eeq{\end{eqnarray}}

\begin{document}
	
	\begin{frontmatter}
	
	\title{ Regional Input-to-State Stability for Systems on Riemannian Manifolds
	}

\date{}
\author[SWJTU]{Li Deng\corref{cor}}
\ead{dengli@swjtu.edu.cn}
\address[SWJTU]{School of Mathematics, Southwest Jiaotong University, Chengdu 611756, Sichuan Province, China}

\author[SWJTU]{Yuan Wang}
\ead{3235642106@qq.com}

\cortext[cor]{Corresponding author.\\
This work was supported by the National Natural Science Foundation of China (Grant No.~12371451) and the Natural Science Foundation of Sichuan Province (Grant No.~2025ZNSFSC0077).
 }

\begin{abstract}
In this paper, we study regional input-to-state stability (regional ISS) for time-varying control systems evolving on connected and complete Riemannian manifolds with a prescribed open state domain.
Using the Riemannian distance, we formulate the notion of a region of ISS and introduce an ISS-Lyapunov function.
We prove that the existence of such a function implies regional ISS in an explicitly estimated geodesic ball and yields a corresponding admissible input bound.
Under additional geometric and dynamical assumptions, this region can be enlarged to the largest geodesic ball centered at the equilibrium point and contained in the prescribed state domain.
Examples under the Euclidean and hyperbolic metrics show that different Riemannian metrics may lead to different estimates for the same system.

	\end{abstract}
	
	
	\end{frontmatter}

\section{Introduction}
Input-to-state stability (ISS) is a fundamental robustness notion for nonlinear systems subject to external inputs or disturbances. In this paper, we study the following problem for time-varying control systems on Riemannian manifolds: given a prescribed open state domain containing an equilibrium point, determine a set of initial states and a bound on the input magnitude such that every corresponding solution remains in the prescribed state domain and satisfies a uniform ISS estimate.

To formulate this problem precisely, let $m,n\in\mathbb N^+$, and let $(M,g)$ be a connected and complete
$n$-dimensional Riemannian manifold; see \cite[p.~11]{lee8}.
We denote by $TM$ the tangent bundle of $M$, and by $T_xM$ the tangent
space of $M$ at $x\in M$.
Let $\rho(\cdot,\cdot)$ be the distance function on $M$ induced by the
Riemannian metric $g$; see \cite[Definition~2.4, p.~146]{c}. Throughout the paper, $|\cdot|$ denotes the Euclidean norm when its
argument belongs to a Euclidean space and the norm induced by $g$ when
its argument is a tangent vector. In particular, for $x\in M$ and $X\in T_xM$,
\[
    |X|=\sqrt{g_x(X,X)}.
\]

Let $L^\infty([0,+\infty);\mathbb R^m)$ denote the space of all
measurable and essentially bounded functions from $[0,+\infty)$ to
$\mathbb R^m$, equipped with the norm
\[
\|u\|_\infty:=\operatorname*{ess\,sup}_{t\geq 0}|u(t)|,
\]
where $|\cdot|$ denotes the Euclidean norm on $\mathbb R^m$.

Let \(D\subseteq M\) be an open set and let $f:[0,+\infty)\times D\times\mathbb R^m\to TM$ be a vector field depending on time and input, that is,
\(f(t,x,u)\in T_xM\) for every \((t,x,u)\). Consider
the following system
\begin{align}\label{sm}
    \dot x(t)=f(t,x(t),u(t)),\quad \text{for a.e. }t\geq t_0,
\end{align}
where $u(\cdot)\in L^\infty([0,+\infty);\mathbb R^m)$ and $t_0\in[0,+\infty)$ are referred to as the \emph{input} and the \emph{initial time}, respectively.
A solution of \eqref{sm} is understood as a locally absolutely continuous curve in $D$ that satisfies \eqref{sm} for a.e. $t$.

Here \(M\) is the ambient state manifold, whereas the open set
\(D\subseteq M\) is a prescribed state domain that forms part of the
data of the problem. The set \(D\) need not be the maximal domain to
which the vector field \(f\) can be extended. It may be the natural
domain on which the model is defined or valid, or it may encode
operational, safety, or feasibility constraints. In either case, only
trajectories remaining in \(D\) are regarded as admissible.

A point $x^*\in D$ is called an \emph{equilibrium point} of system~\eqref{sm} if it satisfies
\begin{align}\label{ep}
    f(t,x^*,0)=0,\qquad \forall\,t\in[0,+\infty).
\end{align}
This leads to the following definition.
\begin{definition}[Regional Input-to-State Stability] 
System \eqref{sm} is called \emph{regionally input-to-state stable (ISS)} in a set $D_1\subseteq D$ containing $x^*$ as an interior point if there exist a constant $\delta>0$, a $\mathcal{KL}$-function $\beta$ whose first-variable domain contains \(\{\rho(x,x^*)\mid x\in D_1\}\), and a $\mathcal K$-function $\gamma:[0,\delta)\to[0,+\infty)$ (see Definition~\ref{kkl1} in the Appendix) such that, for every $t_0\in[0,+\infty)$, every initial condition $x_u(t_0)\in D_1$, and every input $u(\cdot)\in L^\infty([0,+\infty);\mathbb R^m)$ with  
\begin{align}\label{uitz}
\|u\|_{\infty,t_0}
:=
\operatorname*{ess\,sup}_{\tau\geq t_0}|u(\tau)|<\delta,
\end{align}
 every corresponding maximal solution $x_u(\cdot)$ of \eqref{sm}, that is, a
 solution that cannot be extended to a larger time interval as a solution
 taking values in \(D\) (see Definition~\ref{def:maximal-solution}), is defined
 on $[t_0,+\infty)$ and satisfies
\begin{align}\label{iss-ineq}
\rho(x_u(t),x^*)\leq \beta\bigl(\rho(x_u(t_0),x^*),\,t-t_0\bigr) + \gamma\bigl(\|u\|_{\infty,t_0}\bigr), \qquad \forall\,t\geq t_0.
\end{align}
The set \(D_1\) is called a \emph{region of ISS}, and \(\delta\)
is called an \emph{admissible input radius}.
\end{definition}

\begin{remark}\label{rdriss}
The preceding notion is regional with respect to the initial state
and local with respect to the input. More precisely, the initial state
is allowed to range over the prescribed set \(D_1\), whereas the ISS
estimate is required only for inputs satisfying
\(\|u\|_{\infty,t_0}<\delta\).
For the prescribed set \(D_1\), the same functions \(\beta\) and
\(\gamma\) must work for every initial time \(t_0\), every initial
state \(x_u(t_0)\in D_1\), and every admissible input. Thus, the
estimate is uniform with respect to the initial time and the initial
state in \(D_1\), although \(\beta\) and \(\gamma\) may depend on
\(D_1\).
If \(D_1=M\) and the estimate holds for every
\(u(\cdot)\in L^\infty([0,+\infty);\mathbb R^m)\), then it reduces to the
usual notion of global ISS. If \(D_1\) is a sufficiently small
neighborhood of \(x^*\), it corresponds to the usual local ISS property.
\end{remark}

Our objective is to establish sufficient conditions for the preceding regional
ISS property by using ISS-Lyapunov functions, and to estimate both a region of
ISS and a corresponding admissible input radius.

ISS was introduced by Sontag in \cite{MR987806} to describe the robustness of
nonlinear systems with respect to external inputs. Unlike stability properties
for systems without inputs, ISS quantifies in a unified estimate both the decay
of the effect of the initial state and the influence of the input. In
particular, it gives asymptotic stability when the input is zero and bounds the
state in terms of the input magnitude when the input is nonzero. Fundamental
characterizations of ISS,
including characterizations in terms of ISS-Lyapunov functions, were
established in \cite{MR1325675,MR1409473}. These results provide the basic
Lyapunov framework for studying the robustness of nonlinear systems.

There are two main reasons to consider systems with prescribed proper open
state domains. First, ISS on a proper open subset can arise as a subproblem when
studying another problem. For example, Sontag studied ISS of perturbed
gradient flows motivated by model-free feedback control learning and reviewed
ISS for systems evolving on proper open subsets of Euclidean spaces
\cite{MR4374943}. Cui, Jiang and Sontag further studied small-disturbance ISS
of perturbed gradient flows and applied it to policy optimization for the
linear-quadratic regulator problem, where the admissible stabilizing feedback
gains form a proper open subset of a Euclidean space \cite{MR4740203}.

Second, in many applications the state is naturally required to remain in a
prescribed proper open domain. The domain may be determined by the validity of
the model or by operational, safety, or feasibility constraints. For DC
microgrids with ideal constant-power loads, the ISS problem is naturally posed
on a positive-voltage region excluding the singular zero-voltage boundary
\cite{FERGUSON2023110883}. Similarly, biochemical reaction networks and
chemostat models evolve in positive orthants or in the relative interiors of
stoichiometric compatibility classes, and ISS can be used to quantify their
robustness to time-varying reaction or dilution rates
\cite{doi:10.1137/S0363012903437964,MR2293767}. Prescribed regions also encode
operational constraints in power systems, such as bounds on angle differences
and frequency deviations \cite{8386649}. Analogous state restrictions arise
in robotic curve tracking and controlled microelectromechanical relays
\cite{6068226,https://doi.org/10.1002/rnc.1297}. These examples show that the
prescribed state domain is part of the problem itself, rather than a
neighborhood selected only for stability analysis.

When the state domain is a prescribed proper open subset, global ISS on the
ambient state space is generally not the appropriate property, since only
solutions remaining in the prescribed domain are admissible. On the other
hand, a purely local ISS conclusion may provide information only in a small
neighborhood of the equilibrium and need not identify a useful part of the
prescribed domain. This leads naturally to regional ISS. A region of ISS can be regarded as a robust counterpart of the classical region of attraction. For every admissible initial state in this region, the state satisfies a uniform ISS estimate: the contribution of the initial state decays with time, whereas the contribution of the input is bounded in terms of its magnitude. It is
therefore also important to identify an admissible input bound under which the
corresponding solutions remain in the prescribed state domain.

There are several reasons why it is natural to study regional ISS for systems
evolving on Riemannian manifolds.
\begin{enumerate}
\item Many control systems arising from mechanics, robotics, and attitude
dynamics evolve naturally on nonlinear manifolds rather than on Euclidean
spaces. For instance, the configuration space of a mechanical system is
typically a smooth manifold, and the corresponding state space is often its
tangent bundle. Moreover, in mechanical systems, the kinetic energy naturally
induces a Riemannian metric on the configuration manifold; see, for instance,
\cite{jix3}. Thus, the Riemannian distance associated with this metric provides
a natural way to measure the separation between configurations.
\item In the classical ISS theory on Euclidean spaces, the distance from the
state to the equilibrium is measured by the Euclidean norm. On a general
manifold, however, the Euclidean norm is no longer an intrinsic object. The
Riemannian distance provides a natural replacement and leads to a
coordinate-free formulation of regional ISS in terms of the intrinsic
geometry of the state space.
\item Control systems evolving on manifolds are also subject to disturbances,
modeling errors, actuator errors, and measurement noise. Regional ISS provides
a way to quantify their effects while requiring the corresponding solutions
to remain in the prescribed state domain.
\item Stability theory on manifolds is affected by the topology of the state
space. For example, topological obstructions may prevent the existence of a
continuous vector field with a globally asymptotically stable equilibrium; see
\cite{obs}. Thus, regional ISS on manifolds is not merely a
coordinate-dependent reformulation of the Euclidean theory. It requires an
intrinsic treatment that takes into account the differential, metric, and
topological structure of the underlying manifold.
\end{enumerate}

We next review results related to regional ISS, beginning with systems
evolving in Euclidean spaces. Magni, Raimondo and Scattolini established
regional ISS for constrained discrete-time nonlinear systems by using a
nonsmooth ISS-Lyapunov function and applied the result to nonlinear model
predictive control \cite{MR2260135}; related ideas appear in robust model
predictive control \cite{limon2009iss}. For large-scale interconnected
systems, Dashkovskiy and R{\"u}ffer studied local ISS and obtained estimates
for stability regions \cite{MR2642263}. Regional ISS has also been used for
discrete-time fuzzy systems with state delays and saturating actuators, where
the region of attraction and the admissible disturbance energy are estimated
\cite{SilvaEtAl2021}, and for nonlinear systems controlled over communication
networks, where regional ISS is preserved in the presence of time-varying
delays and packet dropouts \cite{PinEtAl2021}. More recently, regional ISS was
used to estimate the domain of attraction and the ultimate bounded region for
finite-control-set model predictive control systems \cite{HuChenWang2025}.

For time-varying nonlinear systems in Euclidean spaces, ISS and corresponding
Lyapunov conditions have been studied in
\cite{EdwardsLinWang2000,MalisoffMazenc2005} and
\cite[Theorem~4.19, p.~176]{jind}. Conditions based on an indefinite upper
bound for the derivative of a Lyapunov function were obtained in
\cite{SahanOzdemir2024}. More recently, Haimovich et al. showed for
time-varying systems on $\mathbb R^n$ that an ISS-Lyapunov function implies
ISS for every maximal solution, even when solutions are nonunique; forward
completeness follows from boundedness and the
boundedness-implies-continuation property
\cite[Proposition~4]{HaimovichLiuRussoMancilla2025}.

For systems evolving on manifolds, Angeli introduced the notion of almost
global ISS, where global stability may be prevented by topological
obstructions \cite{MR2063641}. Input-to-state stabilization on smooth
manifolds by means of sample-and-hold solutions was studied by Malisoff,
Krichman and Sontag \cite{MalisoffKrichmanSontag2006}. Angeli and Efimov
developed ISS-type Lyapunov characterizations for autonomous systems on
Riemannian manifolds with respect to a compact invariant set admitting a
finite decomposition without cycles \cite{AngeliEfimov2015}. Forni and Angeli
studied perturbation and singular perturbation results for
input-to-state multistable systems on Riemannian manifolds
\cite{ForniAngeli2019}. More recently, Mendoza-Avila et al. studied the design
of controllers ensuring ISS or integral ISS for multistable state-periodic
systems and also considered the case where the maximal invariant set of the
closed-loop system is a compact set on a manifold
\cite{MendozaAvilaEtAl2026}. The framework of Angeli and Efimov concerns
global robustness with respect to a possibly multistable invariant set, for
which a classical $\mathcal{KL}$-ISS estimate need not hold. More generally,
the above works concern almost-global stability, multistability, or feedback
stabilization, rather than regional ISS for general time-varying systems on a
prescribed proper open state domain. In contrast, the problem considered here
concerns a single equilibrium and requires a uniform ISS estimate on an
explicitly identified initial-state region together with an admissible input
bound.

In this paper, we formulate regional ISS for system~\eqref{sm} on the
prescribed state domain $D$ and introduce an ISS-Lyapunov function by using the
Riemannian distance. We prove that the existence of an ISS-Lyapunov function
implies regional ISS in an explicitly estimated geodesic ball and also gives a
corresponding admissible input bound; see Theorem~\ref{isslM}. The conclusion
applies to every corresponding maximal solution and does not rely on a
single-valued transition map. Both the initial-state region and the admissible
input bound are determined from the prescribed state domain, the Riemannian
distance and the comparison functions in the Lyapunov conditions. Under
additional geometric and dynamical assumptions, we further show that the
estimated region can be enlarged to the largest geodesic ball centered at the
equilibrium point and contained in $D$; see Theorem~\ref{isse31}. In the
autonomous Euclidean case, Corollary~\ref{global-ISS} recovers the sufficient
direction of the classical ISS-Lyapunov characterization; see
Remark~\ref{rem:euclidean-iss}. The examples also show that, for the same system
and prescribed state domain, different Riemannian metrics may lead to different
estimates of the region of ISS.

As a technical ingredient needed in the proof of our ISS results, we
also establish an intrinsic Carath\'eodory existence and continuation
theorem for time-dependent tangent vector fields on complete Riemannian
manifolds; see Theorem~\ref{cyfx1}. This result is a manifold counterpart
of the classical Euclidean results of Hale
\cite[Theorems~5.1 and~5.2, pp.~28--29]{MR587488}.
It requires only measurability in time, continuity in the state
variable, and local integrable boundedness. In particular, no local
Lipschitz condition, and hence no uniqueness assumption, is imposed.
The resulting continuation alternative is used in the proof of
Theorem~\ref{isslM} to show that a solution whose image remains in a
compact subset of the state domain cannot cease to exist in finite time.

The remainder of this paper is organized as follows. In Section~\ref{smr}, we
present the main results. In Section~\ref{sa5}, we illustrate the results by
examples. In Section~\ref{pmr4}, we prove the main results. In
Appendix~\ref{aur9}, we establish the auxiliary geometric and
Carath\'eodory-ODE results used in the proofs.

\section{Main Results}\label{smr}


Before presenting the main results, we first recall locally absolutely
continuous curves and pointwise velocity vectors, and then state the standing
assumptions and define solutions and maximal solutions of the system.

Motivated by the metric characterization of absolutely continuous
curves in \cite[Definition~1.1.1, p.~23]{MR2129498}, we use the
following local version.

\begin{definition}\label{jDlx}
Let \(c_1,c_2\in \mathbb R\cup\{-\infty,+\infty\}\) with \(c_1<c_2\).
A map \(x:(c_1,c_2)\to M\) is called \emph{locally absolutely continuous}
if, for every compact interval \([t_1,t_2]\subset (c_1,c_2)\), there exists
a function \(m_{t_1,t_2}\in L^1(t_1,t_2)\) such that
\[
    \rho(x(s),x(t))
    \leq \int_s^t m_{t_1,t_2}(\tau)\,d\tau,
    \qquad
    t_1\leq s\leq t\leq t_2.
\]
A map \(x:[a,b]\to M\) is called \emph{absolutely continuous} if there
exists a function \(m\in L^1(a,b)\) such that
\[
    \rho(x(s),x(t))
    \leq \int_s^t m(\tau)\,d\tau,
    \qquad a\leq s\leq t\leq b.
\]
A curve on a half-open interval is called locally absolutely continuous if
its restriction to every compact subinterval is absolutely continuous.
\end{definition}

For a \(C^1\) map \(h:P\to N\) between smooth manifolds, we denote its
differential at \(x\in P\) by
$
dh_x:T_xP\to T_{h(x)}N$; see
\cite[Definition~8.23, p.~259]{gallier-quaintance}.

We shall also use the pointwise velocity vector of a curve on a smooth
manifold. For a \(C^1\) curve \(\gamma:J\to M\), its velocity vector at
\(t_0\in J\) is defined by
\begin{align}\label{dgtzs}
    \dot\gamma(t_0)
    :=d\gamma_{t_0}\left(\left.\frac{d}{dt}\right|_{t_0}\right)
    \in T_{\gamma(t_0)}M;
\end{align}
see \cite[Definition~8.24, p.~261]{gallier-quaintance}. In local coordinates
\(\varphi=(x^1,\ldots,x^n):U\to\mathbb R^n\), with
\(\varphi\circ\gamma=(\gamma^1,\ldots,\gamma^n)\), it is expressed as
\begin{align}\label{ledgtz}
     \dot\gamma(t_0)
    =
    \sum_{i=1}^n
    \frac{d\gamma^i}{dt}(t_0)
    \left.\frac{\partial}{\partial x^i}\right|_{\gamma(t_0)}.
\end{align}
This coordinate expression motivates the following extension to a curve whose local coordinate representation is differentiable only at the time under consideration.

\begin{lemma}\label{lem:pointwise-chain-rule}
Let \(I\subset\mathbb R\) be an interval, let \(t_0\) be an interior point of
\(I\), and let \(\gamma:I\to M\) be a curve. Suppose that there exist a coordinate
chart \(\varphi:U\to\mathbb R^n\) about \(x=\gamma(t_0)\) and a neighbourhood
\(J\subset I\) of \(t_0\) such that \(\gamma(J)\subset U\) and
\(\varphi\circ\gamma\) is differentiable at \(t_0\). Using the canonical
identification \(T_{\varphi(x)}\mathbb R^n\simeq\mathbb R^n\), define
\begin{align}\label{pointwise-velocity}
 \dot\gamma(t_0):=
 d(\varphi^{-1})_{\varphi(x)}
 \big((\varphi\circ\gamma)'(t_0)\big)\in T_xM.
\end{align}
Then \(\dot\gamma(t_0)\) is independent of the chosen coordinate chart. We call
\(\dot\gamma(t_0)\) the \emph{velocity vector} of \(\gamma\) at \(t_0\). If
\(\gamma\) is of class \(C^1\), this definition agrees with the usual
definition of its velocity vector. Moreover, let \(V\subset M\) be any open
neighbourhood of \(x\), and let \(h\in C^1(V)\). Then \(h\circ\gamma\) is
defined in a neighbourhood of \(t_0\), is differentiable at \(t_0\), and
\begin{align}\label{pointwise-chain-rule}
 \left.\frac{d}{dt}\right|_{t=t_0}h(\gamma(t))
 =dh_x\big(\dot\gamma(t_0)\big).
\end{align}
\end{lemma}

The proof of Lemma~\ref{lem:pointwise-chain-rule} is given in the Appendix.

\begin{description}
    \item[\textup{(A)}] Assume that the map  $f:[0,+\infty)\times D\times\mathbb R^m\to TM$ satisfies
\begin{itemize}
    \item for each fixed $(x,u)\in D\times\mathbb R^m$, the map $f(\cdot,x,u)$ is measurable;
    \item for each fixed $(t,u)\in[0,+\infty)\times\mathbb R^m$ and each
    coordinate chart \(\varphi:U\to\mathbb R^n\) in \(D\), the map
    \[
    z\longmapsto
    d\varphi_{\varphi^{-1}(z)}
    \bigl(f(t,\varphi^{-1}(z),u)\bigr)
    \]
    is locally Lipschitz continuous on \(\varphi(U)\);
    \item for each fixed $(t,x)\in[0,+\infty)\times D$, the map $f(t,x,\cdot)$ is continuous.
\end{itemize}
    Moreover, for every $R,r>0$, there exists a locally integrable function
    $I_{R,r}\in L^1_{\mathrm{loc}}([0,+\infty))$ such that
    \[
        |f(t,x,u)|\leq I_{R,r}(t)
    \]
    for a.e. $t\in[0,+\infty)$ and for all $(x,u)\in D\times\mathbb R^m$ with
    \[
        \rho(x,x^*)\leq R,\qquad |u|\leq r,
    \]
    where $x^*\in D$ is an equilibrium point of system~\eqref{sm} and $|\cdot|$ denotes the Euclidean norm in $\mathbb{R}^m$.
    Here $|f(t,x,u)|$ denotes the norm of the tangent vector $f(t,x,u)\in T_xM$ induced by $g$.
\end{description}

Next, we define solutions and maximal solutions of system~\eqref{sm}.

\begin{definition}\label{sssm}
Suppose that Assumption~\textup{(A)} holds. Let
$u(\cdot)\in L^\infty([0,+\infty);\mathbb R^m)$ and
$t_0\in[0,+\infty)$. For $T\in(t_0,+\infty]$, a locally absolutely
continuous curve $x:[t_0,T)\to D$ is called a \emph{solution} to
system~\eqref{sm} on $[t_0,T)$ if
\[
\dot{x}(t)=f(t,x(t),u(t)),
\qquad \text{for a.e. }t\in(t_0,T).
\]
\end{definition}

\begin{definition}\label{def:maximal-solution}
A solution $x:[t_0,T)\to D$ of system~\eqref{sm} is called a
\emph{maximal solution} if there do not exist
$\widetilde T\in(t_0,+\infty]$ with $\widetilde T>T$ and a solution
$\widetilde x:[t_0,\widetilde T)\to D$ such that
$\widetilde x(t)=x(t)$ for all $t\in[t_0,T)$.
\end{definition}

The following lemma provides the existence of solutions under Assumption~\textup{(A)}. Its proof is based on Theorem~\ref{cyfx1} in the Appendix and is given in Section~\ref{pmr4}.

\begin{lemma}\label{lem:exist}
Suppose that Assumption~\textup{(A)} holds. Then, for any
\[
u(\cdot)\in L^\infty([0,+\infty);\mathbb R^m)
\quad\text{and}\quad
(t_0,x_0)\in[0,+\infty)\times D,
\]
there exists $\delta>0$ such
that system~\eqref{sm} admits a solution $x(\cdot)$ on
$[t_0,t_0+\delta)$ with $x(t_0)=x_0$.
\end{lemma}


By the standard extension argument, every local solution admits a maximal
extension. Since uniqueness is not assumed, different maximal extensions may
exist; all statements below are understood to hold for every such extension.

For \(x\in M\) and \(r\geq0\), let
\begin{align}\label{metric-ball}
 B_x(r):=\{y\in M\mid \rho(x,y)<r\}.
\end{align}
In particular, \(B_x(0)=\emptyset\).

We also need to introduce the definition of an ISS-Lyapunov function.
\begin{definition}\label{issf} 
Set
$
    r_0=\sup\{r>0\mid B_{x^*}(r)\subseteq D\}$.
When \(r_0=\infty\), every limit as \(r\to r_0^-\) below is
understood as a limit as \(r\to\infty\).
For a \(C^1\) function \(h:M\to\mathbb R\), we denote by \(dh\) its
differential; see \cite[Definition~2.8, p.~10]{c}. For each \(p\in M\),
the map \(dh_p:T_pM\to\mathbb R\) is linear; see
\cite[Proposition~2.9, p.~9]{c}.

Let
 $
   V:[0,+\infty)\times B_{x^*}(r_0)\to\mathbb R$
be of class \(C^1\). We write \(V_t(t,x)\) for the partial derivative
with respect to \(t\), and \(d_xV(t,x)\) for the differential with
respect to \(x\). The function \(V\) is called an
\emph{ISS-Lyapunov function} for system~\eqref{sm} if there exist a
constant \(\delta_V>0\) and \(\mathcal K\)-functions
 \[
     \alpha_1,\alpha_2,\alpha_3:[0,r_0)\to[0,+\infty),
     \qquad
     \alpha_4:[0,\delta_V)\to[0,r_0),
 \]
 such that
 \[
    \alpha_1(r)\leq\alpha_2(r),
    \qquad \forall\,r\in[0,r_0),
\]
and
\begin{align}
   &\alpha_1(\rho(x,x^*))
    \leq V(t,x)
    \leq \alpha_2(\rho(x,x^*)),
    \label{iss1}\\
  &  V_t(t,x)+d_xV(t,x)\bigl(f(t,x,u)\bigr)
    \leq-\alpha_3(\rho(x,x^*)),
    \label{iss2}
 \end{align}
where \eqref{iss1} holds for all
\((t,x)\in[0,+\infty)\times B_{x^*}(r_0)\), and \eqref{iss2} holds for
all
$
    (t,x,u)\in[0,\infty)\times B_{x^*}(r_0)\times\mathbb R^m
$
satisfying
$
    |u|<\delta_V$ and
    $
    \rho(x,x^*)\geq\alpha_4(|u|)$.
Here $|\cdot|$ denotes the Euclidean norm in $\mathbb{R}^m$.
\end{definition}

\begin{theorem}\label{isslM}
    Assume that \textup{(A)} holds, and that $x^*\in D$ is an equilibrium
    point of system~\eqref{sm}. Define
    \[
       r_0=\sup\{r>0\mid B_{x^*}(r)\subseteq D\}.
    \]
    Suppose that there exists an ISS-Lyapunov function $V$ with
    $\mathcal K$-functions $\alpha_1,\alpha_2,\alpha_3,\alpha_4$ and a
    constant $\delta_V>0$ as in Definition~\ref{issf}. Set
    \begin{align}\label{rvhr}
        r_V = \lim_{r \to r_0^-} \alpha_2^{-1} \circ \alpha_1(r), \quad 
        \hat{r}_V =
        \sup\left\{s\in[0,\delta_V)\mid \alpha_4(s)<r_V\right\}.
    \end{align}
    Then $\hat{r}_V>0$. Moreover, system~\eqref{sm} is regionally ISS in
    $B_{x^*}(r_V)$. That is, there exist a $\mathcal{KL}$-function
    $\beta:[0,r_V)\times[0,+\infty)\to[0,+\infty)$ and a
    $\mathcal K$-function
    $\gamma:[0,\hat{r}_V)\to[0,+\infty)$ such that, for every
    $t_0\geq0$, every $x_0\in B_{x^*}(r_V)$, and every
$u(\cdot)\in L^\infty([0,+\infty);\mathbb R^m)$ satisfying
$\|u\|_{\infty,t_0}<\hat r_V$, every corresponding maximal solution $x_u(\cdot)$ of
\eqref{sm} satisfying $x_u(t_0)=x_0$ is defined on $[t_0,+\infty)$ and
satisfies

    \begin{align}\label{isse}
        \rho(x_u(t),x^*)
        \leq \beta\bigl(\rho(x_u(t_0),x^*),t-t_0\bigr)
        +\gamma\bigl(\|u\|_{\infty,t_0}\bigr),
        \quad \forall\,t\geq t_0.
    \end{align}
    Moreover, $\rho(x^*,x_u(t))<r_0$ holds for all $t\in[t_0,+\infty)$.
\end{theorem}
\begin{remark}\label{rissl}
 Theorem \ref{isslM} states that if there exists an ISS-Lyapunov function, then system (\ref{sm}) is regionally ISS in the set $B_{x^*}(r_V)$. Since $\alpha_1(r)\leq\alpha_2(r)$ for all $r\in[0,r_0)$, we have $r_V \leq r_0$, and hence $B_{x^*}(r_V) \subseteq B_{x^*}(r_0)$. In particular, if
   \begin{align}\label{r0r1}
       \lim_{r\to r_0^-}\alpha_1(r)=\lim_{r\to r_0^-}\alpha_2(r),
   \end{align}
   then the assumption that $\alpha_1$ and $\alpha_2$ are $\mathcal K$-functions implies that $r_0=r_V$. For a given ISS-Lyapunov function $V$, if condition (\ref{r0r1}) does not hold, then $r_V<r_0$.
   In Example \ref{ex52}, we estimate the region of ISS via Theorem \ref{isslM}. Observe that in this example, the estimated region $B_{x^*}(r_V)$ is a strict subset of $B_{x^*}(r_0)$; that is, $B_{x^*}(r_V) \subsetneq B_{x^*}(r_0)$.
 \end{remark}

\begin{corollary}\label{global-ISS}
Assume that all the assumptions of Theorem~\ref{isslM} hold.
Suppose, in addition, that \(D=M\), that
\(\delta_V=+\infty\), and that the comparison functions
\(\alpha_1\) and \(\alpha_2\) are 
\(\mathcal K_\infty\)-functions.
Then system~\eqref{sm} is globally ISS (see Remark \ref{rdriss}).
\end{corollary}

\begin{remark}\label{rem:euclidean-iss}
In the autonomous Euclidean case \(M=\mathbb R^n\), the preceding
corollary recovers the sufficient direction of the classical
ISS-Lyapunov characterization of Sontag and
Wang~\cite{MR1325675}. The converse direction established in that
work, namely that global ISS implies the existence of a global
ISS-Lyapunov function, is not addressed here. Our main concern is
instead the regional problem in which the vector field is prescribed
on an open state domain and both a region of admissible initial states
and a uniform admissible input bound must be determined.
\end{remark}

The preceding corollary addresses the global case \(D=M\). We now return to
the case \(r_0<+\infty\). As observed in Remark~\ref{rissl}, the region
\(B_{x^*}(r_V)\) certified by Theorem~\ref{isslM} may be a strict subset of
\(B_{x^*}(r_0)\). Under additional geometric and dynamical conditions, this
certified region can be enlarged to \(B_{x^*}(r_0)\). To formulate these
conditions, we first introduce some notation from Riemannian geometry. For a
given \(x\in M\), denote by \(O_x\) the origin of the tangent space \(T_xM\),
and, for \(\epsilon>0\), define
   \begin{align}\label{tangent-space-ball}
      B(O_x,\epsilon)=\{X\in T_xM\mid |X|<\epsilon\},
  \end{align}
   where $|X|$ denotes the norm of the vector $X\in T_xM$ induced from the Riemannian metric $g$.
We also denote by \(\exp_x:T_xM\to M\) the exponential map at \(x\);
see \cite[Section~5.5, p.~130]{p1}. Its differential at \(O_x\),
\[
d(\exp_x)_{O_x}:T_{O_x}(T_xM)\to T_xM,
\]
is nonsingular, where \(T_{O_x}(T_xM)\) is naturally isomorphic to \(T_xM\);
see \cite[Proposition~18, p.~131]{p1}. Hence, for sufficiently small
\(\epsilon>0\), the map \(\exp_x\) is a diffeomorphism from
\(B(O_x,\epsilon)\) onto \(B_x(\epsilon)\); see
\cite[Theorem~14, p.~132]{p1}. The latter set is called a
\emph{geodesic ball} of radius \(\epsilon\) centred at \(x\); see
\cite[p.~158]{lee8}.

The injectivity radius \(i(x)\) at \(x\) is the supremum of all
\(\epsilon>0\) such that
\(\exp_x:B(O_x,\epsilon)\to B_x(\epsilon)\) is a diffeomorphism; see
\cite[Section~5.9, p.~142]{p1}.

\begin{theorem}\label{isse31} Assume that all the assumptions in Theorem \ref{isslM} hold. We also assume $r_0\in(0,+\infty)$ and $r_0\leq i(x^*)$,
and
\begin{align}\label{ha1}
  &  \lim_{r\to r_0^-}\frac{\alpha_1(r)}{r_0^2-r^2}=+\infty,
    \\\label{drf3}& d\rho^2(x,x^*)(f(t,x,u))\leq 0,\quad\forall\,(t,x,u)\in[0,+\infty)\times B_{x^*}(r_0) \times\mathbb R^ m\nonumber
    \\ &\phantom{d\rho^2(x,x^*)(f(t,x,u))\leq 0,\quad\forall\,(t,x,u)}\text{with}\;\rho(x,x^*)\geq\alpha_5(|u|),
\end{align}
where $\alpha_5:[0,+\infty)\to[0,+\infty)$ is a $\mathcal K$-function.
Then, system (\ref{sm}) is regionally ISS in $B_{x^*}(r_0)$.
    
\end{theorem}

\begin{remark}\label{rem:estimating-injectivity-radius}
The exact value of \(i(x^*)\) is not needed in order to apply
Theorem~\ref{isse31}. Since \(r_0\) is determined by the state
domain \(D\), the equilibrium point \(x^*\), and the Riemannian metric, it is sufficient to find a lower bound
\(\iota_*>0\) such that
\[
    r_0\leq \iota_*\leq i(x^*).
\]
Thus, only a lower bound strong enough to verify
\(r_0\leq i(x^*)\) is required. If \(r_0>\iota_*\), this lower
bound alone is not sufficient to verify the required inequality. One
may seek a sharper lower bound or restrict the system and the
ISS-Lyapunov function to a smaller open neighborhood of \(x^*\) for
which \(r_0\leq\iota_*\), provided that the other assumptions of
Theorem~\ref{isse31} remain valid there.

Useful estimates of the injectivity radius under geometric
assumptions are given in \cite[Proposition~2.1]{dl26} and the
references therein. In particular, the Euclidean and hyperbolic
spaces considered below have infinite injectivity radius; see
\cite[Proposition~12.9, p.~352]{lee8}.
\end{remark}

\begin{remark}[Relation with classical Lyapunov functions]
The ISS-Lyapunov function can be regarded as a natural extension of the
classical Lyapunov function to control systems with external inputs. Indeed,
condition (\ref{iss1}) provides positive-definite comparison bounds for
\(V(t,x)\) in terms of the distance from \(x\) to the equilibrium point
\(x^*\).

The main difference lies in the dissipation condition (\ref{iss2}). For a
classical Lyapunov function of the unforced system
\[
    \dot x(t)=f(t,x(t)),
\]
one usually requires
\[
    V_t(t,x)+d_xV(t,x)f(t,x)
    \leq -\alpha_3(\rho(x,x^*))
\]
for all \(x\neq x^*\). In contrast, the ISS-Lyapunov condition only requires
the same strict decrease when the state is sufficiently large compared with
the magnitude of the input, namely when
\[
    |u|<\delta_V,\qquad
    \rho(x,x^*)\geq \alpha_4(|u|).
\]
Thus, the input is allowed to prevent the Lyapunov function from decreasing
near the equilibrium, but only inside a neighborhood whose size is determined
by the input magnitude. This is precisely the mechanism behind input-to-state
stability: the state need not converge to \(x^*\) under a nonzero input, but
it remains ultimately bounded by a function of the input size.

If the vector field does not depend on the input, that is,
\[
    f(t,x,u)=f(t,x),
\]
and if \(V\) is a classical Lyapunov function satisfying (\ref{iss1}) together with the classical
dissipation inequality
\[
    V_t(t,x)+d_xV(t,x)\bigl(f(t,x)\bigr)
    \leq -\alpha_3(\rho(x,x^*)),
\]
then \(V\) is also an ISS-Lyapunov function for the system regarded as an
input system with a dummy input \(u\). Indeed, one may choose any
\(\delta_V>0\) and any \(\mathcal K\)-function
\(\alpha_4:[0,\delta_V)\to[0,r_0)\). The above inequality holds independently
of \(u\), and hence it is stronger than the ISS dissipation condition. Taking
\(u\equiv0\) in the ISS estimate gives
\[
    \rho(x(t),x^*)\leq \beta(\rho(x(t_0),x^*),t-t_0),
\]
which is the usual \(\mathcal{KL}\) estimate for uniform asymptotic
stability of the unforced system on the region certified by
Theorem~\ref{isslM}. In the Euclidean case, this conclusion agrees with
the classical Lyapunov theorem for nonautonomous systems; see
\cite[Theorem~4.9, p.~152]{jind}.
\end{remark}

\section{Applications}\label{sa5}

When \(M=\mathbb R^n\) is endowed with the Euclidean metric,
Theorem~\ref{isslM} reduces to the corresponding result for systems in
Euclidean spaces.
In the following examples, we illustrate how Theorem~\ref{isslM} provides
an estimate of the region of ISS and how, under the additional assumptions
of Theorem~\ref{isse31}, that estimate can be enlarged to the largest
geodesic ball centered at the equilibrium and contained in the prescribed
state domain.
In Examples~\ref{ex52} and~\ref{hex7}, we consider the same system on the
prescribed state domain
\[
    H^2=\{(x_1,x_2)\in\mathbb R^2\mid x_2>0\}.
\]
We first regard \(H^2\) as an open subset of the Euclidean space
\(\mathbb R^2\), and then endow the same state domain with the hyperbolic
metric. This allows us to compare the regions of ISS obtained for the same
dynamics and the same admissible state domain under different Riemannian
metrics.
\begin{example}\label{ex52}
Given $a>0$, consider the system
\begin{equation}\label{els3}
  \left\{
  \begin{array}{l}
      \begin{pmatrix} \dot x_1(t) \\ \dot x_2(t) \end{pmatrix}
      =
      \begin{pmatrix}
          -2(1-e^{-t})x_1(t)x_2(t)^2 \\[6pt]
          (1-e^{-t})x_2(t)\bigl(a^2+x_1(t)^2-x_2(t)^2+u(t)\bigr)
      \end{pmatrix}, \\[4pt]
      \text{for a.e. }t \ge t_0, \\[10pt]
      x_2(t) > 0, \quad \forall\,t \ge t_0,
  \end{array}
  \right.
\end{equation}
where $t_0 \in [0,+\infty)$ and 
$u(\cdot) \in L^\infty([0,+\infty);\mathbb R)$. 
The point $A=(0,a)$ is an equilibrium point.
We estimate the region
of ISS of the system with respect to the Euclidean metric on $\mathbb R^2$.
\end{example}

\begin{proof}[Solution]
First, we perform a coordinate transformation
\begin{align*}
    \xi_1 = x_1, \quad \xi_2 = x_2 - a.
\end{align*}
Under this transformation, system~\eqref{els3} becomes
\begin{equation}\label{els7}
  \left\{
  \begin{array}{l}
      \begin{pmatrix} \dot{\xi}_1(t) \\ \dot{\xi}_2(t) \end{pmatrix}
      =
      \begin{pmatrix}
          -2(1-e^{-t})\xi_1(t)(\xi_2(t)+a)^2 \\[6pt]
          (1-e^{-t})\left(\xi_2(t)+a\right)\bigl(a^2+\xi_1(t)^2-(\xi_2(t)+a)^2+u(t)\bigr)
      \end{pmatrix}, \\[4pt]
      \text{for a.e. }t \ge t_0, \\[10pt]
      \xi_2(t) > -a, \quad \forall\,t \ge t_0,
  \end{array}
  \right.
\end{equation}
where $t_0 \in [0,+\infty)$. The origin $O=(0,0)$ is an equilibrium point. We analyze the input-to-state stability of system~\eqref{els7} instead of system~\eqref{els3}.

Let \(f\) denote the vector field of system~\eqref{els7}, that is,
\begin{align*}
    f(t,\xi,u)&=\begin{pmatrix}
          -2(1-e^{-t})\xi_1(\xi_2+a)^2 \\[6pt]
          (1-e^{-t})(\xi_2+a)(a^2+\xi_1^2-(\xi_2+a)^2+u)
      \end{pmatrix},\\
    &(t,\xi,u)\in[0,+\infty)\times
      \{\xi\in\mathbb R^2\mid \xi_2>-a\}\times\mathbb R.
\end{align*}
For \(\xi\in B_O(a)\), set
\[
    d(\xi)=a^2-|\xi|^2,
\]
and define
\begin{equation}\label{veuc}
    V(t,\xi)
    =
    \frac{(1+e^{-t})|\xi|^2}{d(\xi)},
    \qquad
    (t,\xi)\in[0,+\infty)\times B_O(a).
\end{equation}
Clearly, the vector field $f$ satisfies Assumption~\textup{(A)}.

We show that \(V\) is an ISS-Lyapunov function. Put \(s=e^{-t}\).
A direct computation gives
\begin{align}
    &V_t(t,\xi)+d_\xi V(t,\xi)\bigl(f(t,\xi,u)\bigr)
    \nonumber\\
    &=
    -\frac{|\xi|^2}{d(\xi)^2}
    \Bigl[
    s\,d(\xi)
    +2a^2(1-s^2)(\xi_2+a)(\xi_2+2a)
    \Bigr]
    \nonumber\\
    &\quad
    +
    \frac{2a^2(1-s^2)\xi_2(\xi_2+a)}{d(\xi)^2}\,u.
    \label{vtdxtf}
\end{align}
Assume that
\[
    |u|<\frac{a^2}{2},
    \qquad
    |\xi|\geq\sqrt{2|u|}.
\]
Then \(|u|\leq|\xi|^2/2\). Since \(-a<\xi_2<a\), we have
\[
    \xi_2+a>0,
    \qquad
    |\xi_2|\leq\xi_2+2a.
\]
Moreover,
\begin{align*}
    &\frac{2a^2(1-s^2)\xi_2(\xi_2+a)}{d(\xi)^2}u\\
    &\leq
    \frac{2a^2(1-s^2)(\xi_2+a)|\xi_2||u|}{d(\xi)^2}\\
    &\leq
    \frac{a^2(1-s^2)(\xi_2+a)(\xi_2+2a)|\xi|^2}{d(\xi)^2}.
\end{align*}
Consequently, \eqref{vtdxtf} yields
\begin{align*}
    &V_t(t,\xi)+d_\xi V(t,\xi)\bigl(f(t,\xi,u)\bigr)\\
    &\leq
    -\frac{|\xi|^2}{d(\xi)^2}
    \Bigl[
    s\,d(\xi)
    +a^2(1-s^2)(\xi_2+a)(\xi_2+2a)
    \Bigr].
\end{align*}
Moreover,
\[
    2(\xi_2+a)(\xi_2+2a)-d(\xi)
    =
    \xi_1^2+3(\xi_2+a)^2
    \geq0.
\]
Since the function
\[
    s\longmapsto s+\frac{a^2}{2}(1-s^2)
\]
is concave on $[0,1]$, its minimum is attained at an endpoint. Therefore,
\begin{align*}
    &V_t(t,\xi)+d_\xi V(t,\xi)\bigl(f(t,\xi,u)\bigr)\\
    &\leq
    -\frac{|\xi|^2}{d(\xi)}
    \left[
    s+\frac{a^2}{2}(1-s^2)
    \right]\\
    &\leq
    -c_a\frac{|\xi|^2}{a^2-|\xi|^2},
\end{align*}
where
\[
    c_a=\min\left\{1,\frac{a^2}{2}\right\}>0.
\]
Since \(1\leq1+e^{-t}\leq2\) and
\(a^2-|\xi|^2\leq a^2\), we also have
\[
    \frac{|\xi|^2}{a^2}
    \leq V(t,\xi)
    \leq \frac{2|\xi|^2}{a^2-|\xi|^2}.
\]
Thus, \(V\) is an ISS-Lyapunov function with
\begin{align}\label{euc-a1234}
    \alpha_1(r)
    &=
    \frac{r^2}{a^2},
    &
    \alpha_2(r)
    &=
    \frac{2r^2}{a^2-r^2},
    \nonumber\\
    \alpha_3(r)
    &=
    c_a\frac{r^2}{a^2-r^2},
    &
    \alpha_4(r)
    &=
    \sqrt{2r},
\end{align}
where \(r\in[0,a)\) for \(\alpha_1,\alpha_2,\alpha_3\), and
\(r\in[0,a^2/2)\) for \(\alpha_4\). Thus, \(\delta_V=a^2/2\).

For system~\eqref{els7}, \(r_0=a\). From \eqref{rvhr} and
\eqref{euc-a1234}, we obtain
$
    r_V
    =
    \alpha_2^{-1}(1)
    =
    \frac{\sqrt3 a}{3}$ and $
    \hat r_V
    =
    \frac{a^2}{6}$.
Hence, Theorem~\ref{isslM} implies that system~\eqref{els7} is regionally
ISS in \(B_O(\sqrt3a/3)\), with admissible input radius
\(\hat r_V=a^2/6\). Equivalently, the original system~\eqref{els3} is
regionally ISS in \(B_A(\sqrt3a/3)\) for every input satisfying
\(\|u\|_{\infty,t_0}<a^2/6\).

\end{proof}

\begin{remark}\label{exem6}
For system~\eqref{els3}, Theorem~\ref{isslM} gives regional ISS only in
\(B_A(\sqrt3a/3)\), while Theorem~\ref{isse31} enlarges this region to
\(B_A(a)\). Indeed, for system~\eqref{els7}, \(r_0=a<i(O)=+\infty\), and
\eqref{euc-a1234} gives
\(\alpha_1(r)/(a^2-r^2)=r^2/[a^2(a^2-r^2)]\to+\infty\) as \(r\to a^-\);
hence \eqref{ha1} holds. Set \(\alpha_5(r)=\sqrt r\). If
\((t,\xi,u)\in[0,+\infty)\times B_O(a)\times\mathbb R\) and
\(|\xi|\geq\alpha_5(|u|)\), then \(|u|\leq|\xi|^2\), and, since
\(\rho(\xi,O)=|\xi|\),
\begin{align*}
 d_\xi\rho^2(\xi,O)\bigl(f(t,\xi,u)\bigr)
 &=2(1-e^{-t})(\xi_2+a)
   \left(-|\xi|^2(\xi_2+2a)+\xi_2u\right)\\
 &\leq2(1-e^{-t})(\xi_2+a)|\xi|^2
   \left(-(a+\xi_2)-(a-|\xi_2|)\right)\leq0.
\end{align*}
Thus \eqref{drf3} holds. The other assumptions of Theorem~\ref{isse31} were
verified in the solution of Example~\ref{ex52}. Therefore, that theorem shows
that system~\eqref{els3} is regionally ISS in \(B_A(a)\).
\end{remark}

\begin{example}\label{ex:one-dimensional}
Consider
the time-varying system
\begin{equation}\label{eq:one-dimensional}\left\{\begin{array}{l}
    \dot x(t)
    =
    -\left(\frac14+\frac12\cos t\right)x(t)+u(t),
    \quad
    \text{for a.e. }t\geq t_0,
    \\ 
    x(t)\in(-1,1),\quad \forall\,t\geq t_0
\end{array} \right.
\end{equation}
where \(t_0\geq0\) and
\(u(\cdot)\in L^\infty([0,+\infty);\mathbb R)\).
The origin \(x^*=0\) is an equilibrium point. We estimate the region
of ISS of system~\eqref{eq:one-dimensional}.
\end{example}

\begin{proof}[Solution]
Clearly, the vector field
\[
    f(t,x,u)
    =
    -\left(\frac14+\frac12\cos t\right)x+u
\]
satisfies Assumption~\textup{(A)} for
\[
    (t,x,u)\in[0,+\infty)\times(-1,1)\times\mathbb R.
\]
Since the prescribed state domain is $D=(-1,1)$, we have
\[
    r_0
    =
    \sup\{r>0\mid B_0(r)\subset D\}
    =
    1.
\]

Define
\begin{equation}\label{eq:V-one-dimensional}
    V(t,x)=e^{\sin t}x^2,
    \qquad
    (t,x)\in[0,+\infty)\times(-1,1).
\end{equation}
We have
\begin{equation}\label{eq:V-one-dimensional-bound}
    e^{-1}|x|^2
    \leq
    V(t,x)
    \leq
    e|x|^2,\quad (t,x)\in[0,+\infty)\times(-1,1).
\end{equation}

Moreover, a direct computation gives
\begin{align}
    &V_t(t,x)+d_xV(t,x)\bigl(f(t,x,u)\bigr)
    \nonumber\\
    &=
    e^{\sin t}\cos t\,x^2
    +
    2e^{\sin t}x
    \left[
    -\left(\frac14+\frac12\cos t\right)x+u
    \right]
    \nonumber\\
    &=
    e^{\sin t}
    \left(-\frac12x^2+2xu\right).
    \label{eq:V-one-dimensional-derivative}
\end{align}
Thus,
\[
    2|xu|
    \leq
    \frac14|x|^2,\quad\text{if}\;|u|<\frac18\quad 
    \text{and}\quad
    |x|\geq8|u|.
\]
It follows from~\eqref{eq:V-one-dimensional-derivative} and the above relation that
\begin{align*}
    V_t(t,x)+d_xV(t,x)\bigl(f(t,x,u)\bigr)
    &\leq
    -\frac14e^{\sin t}|x|^2\leq
    -\frac{1}{4e}|x|^2,\quad\forall\,t\in[0,+\infty),
\end{align*}
if \(|u|<1/8\) and \(|x|\geq8|u|\).
Therefore, \(V\) is an ISS-Lyapunov function with
\begin{equation}\label{eq:alpha-one-dimensional}
\begin{aligned}
    \alpha_1(r)&=e^{-1}r^2,
    &
    \alpha_2(r)&=er^2,
    &
    \alpha_3(r)&=\frac{1}{4e}r^2,
    \qquad r\in[0,1),\\
    \alpha_4(r)&=8r,
    &
    \delta_V&=\frac18.
\end{aligned}
\end{equation}
Indeed,
\(\alpha_4:[0,1/8)\to[0,1)\) is a \(\mathcal K\)-function.

We now apply Theorem~\ref{isslM}. Since $\alpha_2^{-1}(s)=\sqrt{\frac{s}{e}}$ for $s\in[0,e)$,
we obtain
\begin{align*}
    r_V
    &=
    \lim_{r\to1^-}
    \alpha_2^{-1}\circ\alpha_1(r)
    =
    \lim_{r\to1^-}\frac{r}{e}
    =
    \frac1e,
\\
    \hat r_V
    &=
    \sup
    \left\{
    s\in\left[0,\frac18\right)
    \,\middle|\,
    8s<\frac1e
    \right\}
    =
    \frac{1}{8e}.
\end{align*}
Hence, Theorem~\ref{isslM} implies that
system~\eqref{eq:one-dimensional} is regionally ISS in
$
    B_0\left(\frac1e\right)
    =
    \left(-\frac1e,\frac1e\right)$.
That is, there exist a \(\mathcal{KL}\)-function \(\beta\) and a
\(\mathcal K\)-function \(\gamma\) such that, for every \(t_0\geq0\), every
\(x_0\in B_0(1/e)\), and every
\(u(\cdot)\in L^\infty([0,+\infty);\mathbb R)\) with
\(\|u\|_{\infty,t_0}<1/(8e)\), the corresponding maximal solution
\(x(\cdot)\) of \eqref{eq:one-dimensional} is defined on
\([t_0,+\infty)\) and satisfies
\[
    |x(t)|
    \leq
    \beta(|x(t_0)|,t-t_0)
    +
    \gamma(\|u\|_{\infty,t_0}),
    \quad
    t\geq t_0.
\]
In addition, \(x(t)\in(-1,1)\) for all \(t\in[t_0,+\infty)\).
\end{proof}

\begin{remark}\label{ex32}
This example does not satisfy condition~\eqref{drf3}. Thus, we cannot apply Theorem~\ref{isse31} to enlarge the certified ISS region from \(B_0(1/e)\) to \(B_0(1)\).

Indeed, taking \(u=0\) and \(t=\pi\), we have
\[
    f(\pi,x,0)
    =
    \frac14x,\quad\forall\,x\in(-1,1).
\]
Then, for every \(x\in(-1,1)\setminus\{0\}\),
\[
    d_x|x|^2\bigl(f(\pi,x,0)\bigr)
    =
    2x\left(\frac14x\right)
    =
    \frac12x^2
    >
    0,
\]
which implies that condition~\eqref{drf3} does not hold. 

Moreover, the system is not regionally ISS in \(B_0(1)\). Indeed, consider
the same differential equation on \(\mathbb R\) with \(u=0\). Let
\(t_0=\frac{2\pi}{3}\), \(t_1=\frac{4\pi}{3}\), and let
\(\bar x(\cdot)\) be its global solution with an initial state satisfying
\begin{align*}
    1>|\bar x(t_0)|> e^{\frac14(t_1-t_0)+\frac12(\sin t_1-\sin t_0)} =e^{\pi/6-\sqrt{3}/2}>e^{-1}.
\end{align*}
Such an initial state exists because
\(-1<\pi/6-\sqrt{3}/2<0\). The explicit solution gives
\begin{align*}
    |\bar x(t_1)|=|\bar x(t_0)|e^{\frac14(t_0-t_1)+\frac12(\sin t_0-\sin t_1)}>1.
\end{align*}
Since \(|\bar x(t_0)|<1\), continuity implies that there exists
\(\tau\in(t_0,t_1)\) such that \(|\bar x(\tau)|=1\). Consequently, the
corresponding maximal solution in the prescribed state domain \((-1,1)\)
cannot be defined on \([t_0,+\infty)\), which contradicts regional ISS in
\(B_0(1)\).
\end{remark}

\begin{example}\label{hex7}
Consider system~\eqref{els3}. Set
$
H^2=\{(x_1,x_2)\in\mathbb R^2\mid x_2>0\}$.
The space \(H^2\) is a smooth manifold with the global coordinate chart
\((H^2,I_{H^2})\), where \(I_{H^2}\) is the restriction of the identity map
\(I:\mathbb R^2\to\mathbb R^2\) to \(H^2\). In this coordinate chart,
\(\{\left.\frac{\partial}{\partial x_1}\right|_x,
\left.\frac{\partial}{\partial x_2}\right|_x\}\) is a basis of \(T_xH^2\)
for every \(x\in H^2\).
Endow $H^2$ with the hyperbolic metric $g$, whose coefficients in this
coordinate chart are
\begin{align}\label{hmh2}
g_{ij}(x)
&=g_x\left(
\left.\frac{\partial}{\partial x_i}\right|_x,
\left.\frac{\partial}{\partial x_j}\right|_x
\right)
=\frac{\delta_{ij}}{x_2^2},
\qquad i,j=1,2,\quad x=(x_1,x_2)\in H^2.
\end{align}
where $\delta_{ij}$ is the Kronecker delta.
We estimate the region
of ISS of system (\ref{els3}) in the hyperbolic metric.
\end{example}

\begin{proof}[Solution]
First, we set $M=D=H^2$ and view system (\ref{els3}) as a system evolving on a Riemannian manifold.
It is shown in \cite[p.~162]{c} that $(H^2, g)$ is a simply connected and complete Riemannian manifold with constant sectional curvature $-1$. According to \cite[Proposition~12.9, p.~352]{lee8}, the injectivity radius of each point of $H^2$ is $+\infty$.

For $x=(x_1,x_2)\in H^2$, set
\begin{align}\label{tftx9}
\tilde f(t,x,u)=\begin{pmatrix}
-2(1-e^{-t})x_1x_2^2 \\
(1-e^{-t})x_2(a^2+x_1^2-x_2^2+u)
\end{pmatrix},\quad\forall\,(t,x,u)\in[0,+\infty)\times H^2\times\mathbb R,
\end{align}
which, as a vector at each $x=(x_1,x_2)\in H^2$, can be expressed in the coordinate chart by
\begin{align*}
\tilde f(t,x,u)=-2(1-e^{-t})x_1x_2^2\frac{\partial}{\partial x_1}+(1-e^{-t})x_2\left( a^2+x_1^2-x_2^2+u\right)\frac{\partial}{\partial x_2}.
\end{align*}
Thus, system (\ref{els3}) can be expressed as
\begin{align}\label{sh21}
\left\{ \begin{array}{l}
\dot x(t)= \tilde f(t,x(t),u(t)),\quad \text{for a.e. } t\in[t_0,+\infty), \\
x(t)\in H^2,\quad\forall\,t\in[t_0,+\infty),
\end{array}
\right.
\end{align}
where $t_0\geq 0$. Clearly, the vector field $\tilde f$ satisfies
Assumption~\textup{(A)}.

Next, we denote by $\rho(x,A)$ the distance between $x\in H^2$ and $A$ in the metric $g$, and present some relevant properties of  $\rho(\cdot,A)$. 
It is shown in \cite[Theorem~33.5.1, p.~614]{for} that \(\rho(\cdot,A)\) is given by
\begin{align}\label{hsd1}
\rho(x,A)  =\ln\frac{\sqrt{x_{1}^{2}+(x_{2}+a)^{2}}+\sqrt{x_{1}^{2}+(x_{2}-a)^{2}}}{\sqrt{x_{1}^{2}+(x_{2}+a)^{2}}-\sqrt{x_{1}^{2}+(x_{2}-a)^{2}}},\quad\forall\,x=(x_1,x_2)\in H^2.
\end{align}

We claim that
\begin{align}\label{drsX1}&
\sqrt{x_1^2+(x_2-a)^2}=2\sqrt{ax_2}\sinh \frac{\rho(x,A)}{2},\quad \sqrt{x_1^2+(x_2+a)^2}=2\sqrt{ax_2}\cosh \frac{\rho(x,A)}{2},
\\ \label{drsX3}&
\sqrt{x_1^2+(x_2-a)^2}\sqrt{x_1^2+(x_2+a)^2}= 2a x_2\sinh\rho(x,A),
\\\label{drsX5}& 
a e^{-\rho(x,A)}\leq x_2\leq a e^{\rho(x,A)},
\\ \label{drsX}&
d\rho^2(x,A)(X)=
\begin{cases}
\displaystyle
\frac{2\rho(x,A)\left(2x_1 X_1+x_2^{-1}(x_2^2-x_1^2-a^2)X_2\right)}
{\sqrt{x_{1}^{2}+(a+x_{2})^{2}}\sqrt{x_{1}^{2}+(a-x_{2})^{2}}},
& x\neq A,\\[3mm]
0, & x=A,
\end{cases}\nonumber
\\&\quad\quad\quad\quad\quad
\quad\quad\quad\quad\quad\quad\quad
\forall\,X=X_1\frac{\partial}{\partial x_1}+X_2\frac{\partial}{\partial x_2}\in T_x H^2,
\end{align}
hold for all $x=(x_1,x_2)\in H^2$.

In fact, for any \(x=(x_1,x_2)\in H^2\), set
\begin{align*}
 \ell_-=\sqrt{x_1^2+(x_2-a)^2},\quad \ell_+=\sqrt{x_1^2+(x_2+a)^2}.
\end{align*}
By direct computations, we obtain, via (\ref{hsd1}), that
\begin{align}\label{eruv}
 e^{\rho(x,A)}=\frac{\ell_++\ell_-}{\ell_+-\ell_-},\quad \ell_+^2=\ell_-^2+4ax_2,
\end{align}
which implies 
\begin{align*}
 \ell_-^2=4ax_2\left( \frac{e^{\rho(x,A)/2}-e^{-\rho(x,A)/2}}{2}\right)^2=4ax_2\left(\sinh \frac{\rho(x,A)}{2}\right)^2.
\end{align*}
Thus, we obtain the first equality of (\ref{drsX1}). The second equality of (\ref{drsX1}) follows directly from the second equality of (\ref{eruv}) and the first equality of (\ref{drsX1}). Equality (\ref{drsX3}) is deduced from (\ref{drsX1}). 

To prove (\ref{drsX5}), we obtain from the first equality of (\ref{drsX1}) that
\begin{align*}
x_1^2+x^2_2+a^2=2ax_2\cosh\rho(x,A),
\end{align*}
which implies
\begin{align*}
x_1^2+(x_2-a\cosh\rho(x,A))^2=a^2\left((\cosh\rho(x,A))^2-1  \right)=a^2 (\sinh\rho(x,A))^2.
\end{align*}
Thus, we have
\begin{align*}
-a\sinh\rho(x,A)\leq x_2-a\cosh\rho(x,A)\leq a\sinh\rho(x,A).
\end{align*}
Then, (\ref{drsX5}) follows.

Finally, we show \eqref{drsX}. By \eqref{dhnh}, \eqref{dual-covector} and \eqref{rho2-gradient}, we obtain that 
\begin{align}\label{drho2-exp-inverse}
d\rho^2(x,A)(X)=-2g_x(\exp_x^{-1}A,X),\quad\forall\,X\in T_x H^2.
\end{align}
Formula~(2.16) in \cite{dl26} provides the expression for $\exp_x^{-1}A$ in the coordinate chart $(H^2, I_{H^2})$, which, together with (\ref{hmh2}), gives the first case in (\ref{drsX}) for $x\neq A$.
For $x=A$, since $\exp_A^{-1}A=O_A$, identity \eqref{drho2-exp-inverse} gives $d\rho^2(A,A)(X)=0$ for all $X\in T_AH^2$.
Moreover, the first case has this value as its continuous extension at $A$: indeed, by (\ref{drsX3}), for $x\neq A$ it can be rewritten as
\[
\frac{\rho(x,A)}{a x_2\sinh\rho(x,A)}
\left(2x_1X_1+x_2^{-1}(x_2^2-x_1^2-a^2)X_2\right),
\]
which tends to $0$ as $x\to A$, because $\rho(x,A)/\sinh\rho(x,A)\to1$, $x_2\to a$, and the expression in parentheses tends to $0$.
This completes the proof of the claim.

Now, we show that the function 
\begin{align}\label{islh1}
V(t,x)=(1+e^{-t})\rho^2(x,A),\quad\forall\,(t,x)\in [0,+\infty)\times H^2
\end{align}
is an ISS-Lyapunov function. 

To this end, by (\ref{drsX}) and the definition of $V$, for $x\neq A$ we have
\begin{align}\label{hzr4}
\begin{array}{ll}
& V_t(t,x)+d_xV(t,x)(\tilde f(t,x,u)) \\
&\displaystyle = -e^{-t}\rho^2(x,A)+2(e^{-2t}-1)\rho(x,A)\sqrt{(x_1^2+(a+x_2)^2)(x_1^2+(a-x_2)^2)} \\
&\displaystyle \quad +\frac{2(1-e^{-2t})\rho(x,A)u(x_2^2-x_1^2-a^2)}{\sqrt{(x_1^2+(a+x_2)^2)(x_1^2+(a-x_2)^2)}},
\end{array}
\end{align}
where $(t,x,u)\in[0,+\infty)\times (H^2\setminus\{A\})\times\mathbb R$.

Set
\begin{align*}
h(x)&=\min\left\{  \rho(x,A), \sqrt{(x_1^2+(a+x_2)^2)(x_1^2+(a-x_2)^2)}\right\},\quad\forall\,x=(x_1,x_2)\in H^2, \\
\zeta(r)&=\min\{  r,a^2(1-e^{-2r})\},\quad \forall\,r\in[0,+\infty).
\end{align*}
We obtain from (\ref{drsX3}) and (\ref{drsX5}) that
\begin{align}\label{hzr1}
h(x)\geq\zeta(\rho(x,A)),\quad\forall\,x\in H^2.
\end{align}
By direct computations, we have
\begin{align*}
(x_1^2+(a+x_2)^2)(x_1^2+(a-x_2)^2)=(x_2^2-x_1^2-a^2)^2+4x_1^2x_2^2,\quad\forall\,(x_1,x_2)\in H^2.
\end{align*}

It follows directly from the definition that
\(\zeta:[0,+\infty)\to[0,a^2)\) is a \(\mathcal K\)-function. Consequently,
by \cite[Lemma~4.2, p.~145]{jind}, the function
\(\zeta^{-1}(\frac{\cdot}{b}):[0,a^2b)\to[0,+\infty)\) is also a
\(\mathcal K\)-function for any \(b>0\).

Now, choose $b\in(0,1)$. Take any $(t,x,u)\in [0,+\infty)\times H^2\times \mathbb R$ with $|u|<a^2b$ and $\rho(x,A)\geq \zeta^{-1}(\frac{|u|}{b})$. If $x=A$, then this condition forces $u=0$, and $V_t(t,A)+d_xV(t,A)(\tilde f(t,A,0))=0$. If $x\neq A$, it follows from the above relation, (\ref{hzr4}) and (\ref{hzr1}) that
\begin{align*}
& V_t(t,x)+d_xV(t,x)(\tilde f(t,x,u)) \\
&\leq(2e^{-2t}-e^{-t}-2)\rho(x,A)h(x)+2(1-e^{-2t})\rho(x,A)|u| \\
&\leq (2e^{-2t}-e^{-t}-2)\rho(x,A)\zeta(\rho(x,A))+2b(1-e^{-2t})\rho(x,A)\zeta\left(\rho(x,A) \right) \\
&=-\left(-2(1-b)s^2+s-2b+2 \right)\rho(x,A)\zeta(\rho(x,A)) \\
&\leq- \min\{2-2b,1\}\rho(x,A)\zeta(\rho(x,A)),
\end{align*}
where $s=e^{-t}$. Since $b\in(0,1)$, the function
\(r\mapsto\min\{2-2b,1\}r\zeta(r)\), \(r\in[0,+\infty)\), is a
\(\mathcal K\)-function. Thus, $V$ is an ISS-Lyapunov function with
\(\delta_V=a^2b\) and the following \(\mathcal K\)-functions:
\begin{align}\label{a1234}
\begin{array}{ll}
\alpha_1(r)&=r^2,\quad \alpha_2(r)=2r^2, \quad\alpha_3(r)=\min\{2-2b,1\}r\zeta(r),\quad\forall\,r\in[0,+\infty), \\[2mm]
\alpha_4(r)&=\zeta^{-1}\left(\displaystyle\frac{r}{b}\right),\quad r\in [0,a^2b).
\end{array}
\end{align}

Finally, we apply Theorem \ref{isslM} to system (\ref{sh21}). Recalling the definitions of $r_V$ and $\hat{r}_V$ (see (\ref{rvhr})), we obtain in this case that
\begin{align*}
r_V=+\infty, \quad \hat{r}_V=a^2b.
\end{align*}
By Theorem \ref{isslM}, there exist a $\mathcal{KL}$-function $\beta: [0,+\infty)\times[0,+\infty)\to[0,+\infty)$ and a $\mathcal K$-function $\gamma: [0,a^2b)\to [0,+\infty)$ such that, for every $t_0\geq 0$, every initial condition $x_0\in H^2$, and every input $u(\cdot)\in L^\infty([0,+\infty);\mathbb R)$ with $\|u\|_{\infty,t_0}<a^2b$, every corresponding maximal solution $x(\cdot)$ of system (\ref{sh21}) satisfying $x(t_0)=x_0$ is defined on $[t_0,+\infty)$ and satisfies
\begin{align*}
\rho(x(t),A)\leq\beta\left(\rho(x(t_0), A), t-t_0 \right)+\gamma\left( \|u\|_{\infty,t_0} \right),\quad \forall\,t\geq t_0.
\end{align*}
Thus, system (\ref{els3}) is regionally ISS in $H^2$, with admissible input radius $a^2b$.
\end{proof}

\begin{remark}\label{exem1}
For both the Euclidean and hyperbolic metrics considered above, an
ISS-Lyapunov function for system~\eqref{els3} must depend on time. Indeed,
the vector fields in \eqref{els7} and \eqref{sh21} satisfy
\[
\begin{aligned}
 f(0,\xi,0)&=O_\xi,
 &&\forall\,\xi\in B_O(a)\setminus\{O\},\\
 \widetilde f(0,x,0)&=O_x,
 &&\forall\,x\in H^2\setminus\{A\},
\end{aligned}
\]
where \(O_\xi\in T_\xi\mathbb R^2\) and \(O_x\in T_xH^2\) denote the
corresponding zero vectors.
If an ISS-Lyapunov function \(W\) were independent of time, then, since
\(\alpha_4(0)=0\), \eqref{iss2} applied with \(t=0\) and \(u=0\) would give
\[
 0=dW(q)(O_q)\leq-\alpha_3\bigl(\rho(q,q^*)\bigr)<0,
\]
where \((q,q^*)=(\xi,O)\) in the Euclidean case with \(\xi\in B_O(a)\setminus\{O\}\), and
\((q,q^*)=(x,A)\) in the hyperbolic case with \(x\in H^2\setminus\{A\}\). This is a contradiction.
\end{remark}

\section{Proof of the main results}\label{pmr4}

First, we prove Lemma \ref{lem:exist}.
\begin{proof}[Proof of Lemma \ref{lem:exist}]
Choose a measurable representative of $u(\cdot)$, still denoted by $u(\cdot)$.
Set $b(t,x)=f(t,x,u(t))$ for $t\geq0$ and extend $b$ to negative
times by reflection, namely by setting $b(t,x)=b(-t,x)$ for $t<0$.
For each fixed \(x\in D\), the map
\((t,v)\mapsto f(t,x,v)\) is a Carath\'eodory map.
Since \(u(\cdot)\) is measurable, it follows from
\cite[Lemma~8.2.3, p.~311]{AF} that
\(t\mapsto f(t,x,u(t))\) is measurable.
Hence, $t\mapsto b(t,x)$ is measurable. Moreover, by
Assumption~\textup{(A)}, the map $x\mapsto b(t,x)$ is continuous for every
fixed $t\in\mathbb R$.

Let $J\subset\mathbb R\times D$ be compact. There exists $R>0$ such that
$\rho(x,x^*)\leq R$ whenever $(t,x)\in J$. Choose
$r>\|u\|_\infty$. Then, by Assumption~\textup{(A)},
\[
    |b(t,x)|\leq I_{R,r}(|t|)
\]
for a.e. $t\in\mathbb R$ and all $x\in D$ such that $(t,x)\in J$.
Since $t\mapsto I_{R,r}(|t|)$ belongs to
$L^1_{\mathrm{loc}}(\mathbb R)$, the extended vector field satisfies all
the hypotheses of Theorem~\ref{cyfx1}. That theorem yields $\delta>0$ and
a local solution through $(t_0,x_0)$ on
$(t_0-\delta,t_0+\delta)$. Its restriction to
$[t_0,t_0+\delta)$ is a solution of system~\eqref{sm} with
$x(t_0)=x_0$.
\end{proof}

We next prove Theorem~\ref{isslM}.

\begin{proof}[Proof of Theorem \ref{isslM}]
Throughout the proof, for a given initial condition $(t_0,x_0)$ and an input $u(\cdot)$, 
we denote by $x_u(\cdot;t_0,x_0)$ (or simply $x_u(\cdot)$ when clear from context) 
every corresponding solution to system \eqref{sm}.
The proof is divided into seven steps.

{\it Step 1.}\;Set $a_1=\lim_{r\to r_0^-}\alpha_1(r)$. From (\ref{iss1}), we have $a_1\leq\lim_{r\to r_0^-}\alpha_2(r)$. Since $\alpha_2:[0,r_0)\to\mathbb R$ is a $\mathcal K$-function, it follows from \cite[Lemma~4.2, p.~145]{jind} that its inverse $\alpha_2^{-1}:[0,a_1)\to\mathbb R$ is also a $\mathcal K$-function, with $\alpha_2^{-1}([0,a_1))\subset[0,r_0)$. Moreover, by the same lemma, $\alpha_3\circ\alpha_2^{-1}:[0,a_1)\to\mathbb R$ is a $\mathcal K$-function. By \cite[Footnote~20, p.~153]{jind}, there exists a locally Lipschitz $\mathcal K$-function $\widetilde\alpha:[0,a_1)\to[0,+\infty)$ such that
\begin{align}\label{tilde-alpha}
 \widetilde\alpha(s)\leq \alpha_3\circ\alpha_2^{-1}(s),
 \qquad \forall\,s\in[0,a_1).
\end{align}

We next consider the following ordinary differential equation
\begin{align}\label{odode}
 \begin{cases}  
 \dot y(t)=-\widetilde\alpha(y(t)), & \forall\,t\geq t_0, \\
 y(t_0)=y_0, 
 \end{cases}
\end{align}
where $(t_0,y_0)\in[0,+\infty)\times[0,a_1)$.  Denote by 
$y(\cdot;t_0,y_0)$ the solution to the above equation with initial time $t_0$ and initial condition $y_0$.
It follows from \cite[Lemma~4.4, p.~145]{jind} that there exists a $\mathcal{KL}$-function $\beta_0: [0,a_1)\times[0,+\infty)\to[0,+\infty)$ such that the solution of (\ref{odode}) satisfies
\begin{align*}
    y(t;t_0,y_0)=\beta_0(y_0,t-t_0),\quad\forall\,(y_0,t_0)\in[0,a_1)\times [0,+\infty),\; t\geq t_0.
\end{align*}

{\it Step 2}.\; We claim that 
\begin{align}\label{alc}
  &  \alpha_2(\hat c)<a_1,\quad\forall\,\hat c\in[0,r_V),
    \\ \label{bais}&\operatorname{cl}\,B_{x^*}\left(\alpha_2^{-1}(c)\right)\subseteq S_{c,t},\quad\forall\,(c,t)\in[0,a_1)\times[0,+\infty),
\end{align}
where for each $c\in[0,a_1)$, 
\begin{align}\label{sct}
    S_{c,t}=\{x\in B_{x^*}(r_0)\mid V(t,x)\leq c\},\quad\forall\,t\in[0,+\infty),
\end{align}
and $\operatorname{cl}\,A$ denotes the closure of  $A\subset M$.

In fact, by \cite[Lemma~4.2, p.~145]{jind}, the function
$\alpha_2^{-1}\circ\alpha_1:[0,r_0)\to\mathbb R$ is a
$\mathcal K$-function. For any
$\hat c\in[0,r_V)=[0,\lim_{r\to r_0^-}\alpha_2^{-1}\circ\alpha_1(r))$,
there exists $r_{\hat c}\in[0,r_0)$ such that
$\hat c=\alpha_2^{-1}\circ\alpha_1(r_{\hat c})$. Then
$\alpha_2(\hat c)=\alpha_1(r_{\hat c})<a_1$. Thus, we obtain (\ref{alc}).

To show (\ref{bais}), we take any $c\in[0,a_1)$. Since $\alpha_1$ is continuous and strictly increasing, and satisfies $\alpha_1(0)=0$,
we have
$
\alpha_1([0,r_0))=[0,a_1)$.
Therefore, there exists $r_c\in[0,r_0)$ such that
$c=\alpha_1(r_c)$.
 Thus, we have $\alpha_2^{-1}(c)=\alpha_2^{-1}\circ\alpha_1(r_c)<r_V$, which, together with Remark \ref{rissl},  implies 
\begin{align*}
    \operatorname{cl}\, B_{x^*}\left( \alpha_2^{-1}(c) \right)\subset B_{x^*}(r_V)\subseteq B_{x^*}(r_0)\subseteq D.
\end{align*}

When $c>0$, $\alpha_2^{-1}(c)>0$. For any $x\in \operatorname{cl}\, B_{x^*}\left( \alpha_2^{-1}(c) \right)$, it follows from (\ref{iss1}) that 
\begin{align*}
    V(t,x)\leq \alpha_2(\rho(x,x^*))\leq \alpha_2\left(\alpha_2^{-1}(c)  \right)=c,\quad\forall\,t\in[0,+\infty),
\end{align*}
which implies $x\in S_{c,t}$ for all $t\geq 0$. Thus we obtain $\operatorname{cl}\, B_{x^*}\left( \alpha_2^{-1}(c) \right)\subseteq S_{c,t}$.
When $c=0$, $\alpha_2^{-1}(0)=0$, it follows from (\ref{iss1}) that  $S_{0,t}=\{x^*\}$ for all $t\geq 0$,  and consequently $\operatorname{cl}\, B_{x^*}\left( \alpha_2^{-1}(0) \right)=\emptyset\subset S_{0,t}$ for all $t\geq 0$. We obtain (\ref{bais}).

{\it Step 3.}\;Since $r_V>0$ and $\alpha_4(\cdot)$ is continuous and strictly increasing with $\alpha_4(0)=0$, we have
$\hat{r}_V>0$.
Let
$
    u(\cdot)\in L^\infty([0,+\infty);\mathbb R^m)$ and $
    t_0\in[0,+\infty)$, 
and assume that
\begin{equation}\label{uhat}
    \|u\|_{\infty,t_0}<\hat{r}_V.
\end{equation}
Here $\|u\|_{\infty,t_0}$ is defined by (\ref{uitz}).
It follows from the definition of $\hat{r}_V$ that
\begin{equation}\label{a4u}
    \|u\|_{\infty,t_0}<\delta_V,
    \qquad
    \alpha_4\left(\|u\|_{\infty,t_0}\right)<r_V.
\end{equation}

We claim that, if the initial state
$x_u(t_0)\in B_{x^*}(r_V)$, then every corresponding maximal solution
$x_u(\cdot)$ to system (\ref{sm}) is defined on
$[t_0,+\infty)$ and satisfies
\begin{align}\label{xubr0i}
    x_u(t)\in B_{x^*}(r_0),
    \qquad \forall\,t\in[t_0,+\infty).
\end{align}

To prove this claim, we first take any corresponding maximal solution
\[
    x_u:[t_0,T_u)\longrightarrow D
\]
where
\(T_u\in(t_0,+\infty]\) is the right endpoint of its interval of
existence. Set
\[
q:=
\max\left\{
\alpha_4\left(\|u\|_{\infty,t_0}\right),
\rho(x^*,x_u(t_0))
\right\}.
\]
Then
by (\ref{a4u}) and $x_u(t_0)\in B_{x^*}(r_V)$, we have
\begin{align}\label{rqrv}
    \rho(x^*, x_u(t_0))\leq  q<r_V.
\end{align}
It follows from (\ref{alc}) that
\[
    \alpha_2(q)<a_1.
\]
Hence, there exists a real number $r_u$ such that
\begin{align}\label{ruqr0}
  r_u\in(q,r_0)\quad\text{and}\quad 
    \alpha_2(q)<\alpha_1(r_u)<a_1.  
\end{align}

We show that
\begin{align}\label{rxut0tu}
    \rho(x^*,x_u(t))<r_u,
    \qquad \forall\,t\in[t_0,T_u).
\end{align}
Suppose, to the contrary, that this assertion is false. By the continuity of
$x_u(\cdot)$ and the distance function, the set
$
\{t\in(t_0,T_u)\mid \rho(x^*,x_u(t))=r_u\}
$ 
is nonempty. Set
\[
T:=\inf\{t\in(t_0,T_u)\mid \rho(x^*,x_u(t))=r_u\}.
\]
Then \(T\in(t_0,T_u)\) and
\[
    \rho(x^*,x_u(T))=r_u.
\]
Since
\[
    \rho(x^*,x_u(t_0))
    \leq q<r_u
    =\rho(x^*,x_u(T)),
\]
the intermediate value theorem implies that the set
\[
    \left\{
    t\in[t_0,T]\mid \rho(x^*,x_u(t))=q
    \right\}
\]
is nonempty. This set is closed in the compact interval $[t_0,T]$.
Thus, we may define
\[
    \hat T_u
    =
    \max\left\{
    t\in[t_0,T]\mid \rho(x^*,x_u(t))=q
    \right\}.
\]
Since $\rho(x^*,x_u(T))=r_u>q$,
we have $\hat T_u<T$.
  Moreover,
\[
    \rho(x^*,x_u(t))>q
    \geq
    \alpha_4\left(\|u\|_{\infty,t_0}\right),
    \qquad
    \forall\,t\in(\hat T_u,T].
\]

Since
\[
    |u(t)|\leq\|u\|_{\infty,t_0}<\delta_V,\quad a.e.\,t\geq t_0,
\]
 the monotonicity of $\alpha_4$ gives
\[
    \rho(x^*,x_u(t))
    >
    \alpha_4\left(\|u\|_{\infty,t_0}\right)
    \geq
    \alpha_4(|u(t)|),\quad a.e.\,t\in(\hat T_u,T).
\]
Therefore, it follows from
(\ref{iss2}) that
\[
    \frac{d}{dt}V(t,x_u(t))
    \leq
    -\alpha_3\left(\rho(x^*,x_u(t))\right)
    \leq0,\quad a.e.\,t\in(\hat T_u,T).
\]

Since $x_u(\cdot)$ is locally absolutely continuous and $V$ is of
class $C^1$, the function
\[
    t\longmapsto V(t,x_u(t))
\]
is locally absolutely continuous. Consequently, by (\ref{ruqr0}), we have
\begin{align*}
    \alpha_1(r_u)
    \leq V(T,x_u(T))
    \leq V(\hat T_u,x_u(\hat T_u))
    \leq \alpha_2(q)
    <\alpha_1(r_u),
\end{align*}
which is a contradiction. Therefore, we obtain (\ref{rxut0tu}).

We next show
\begin{align}\label{tu=i}
    T_u=+\infty.
\end{align}
Suppose, to the contrary, that this assertion is false. Then, $T_u<+\infty$. By (\ref{rxut0tu}), $x_u(t)\in \operatorname{cl}B_{x^*}(r_u)$ for all $t\in [t_0,T_u)$.
Since $M$ is complete, the Hopf--Rinow theorem
\cite[Theorem~2.8, p.~146]{c} implies that
$\operatorname{cl}B_{x^*}(r_u)$ is compact. Since $r_u<r_0$ (see (\ref{ruqr0})), we have
\[
\operatorname{cl}B_{x^*}(r_u)\subset B_{x^*}(r_0)\subseteq D.
\]
By the maximality of \(x_u(\cdot)\), it cannot be extended beyond \(T_u\).
This contradicts assertion \textup{(ii)} of
Theorem~\ref{cyfx1}. Therefore, we obtain (\ref{tu=i}).

Finally, estimate (\ref{rxut0tu}) gives (\ref{xubr0i}).
This proves the claim.

{\it Step 4.}\;Let the input $u(\cdot)\in L^\infty([0,+\infty);\mathbb R^m)$ and the initial time $t_0\in[0,+\infty)$ satisfy
(\ref{uhat}). Then (\ref{a4u}) holds. Assume the initial state $x_u(t_0) \in B_{x^*}(r_V)$.
Let $x_u(\cdot)$ be a solution to system (\ref{sm}) corresponding to the input $u(\cdot)$, the initial time $t_0$ and the initial state $x_u(t_0)$. It follows from  Step 3 that $x_u(t)\in B_{x^*}(r_0)$ for $t\in[t_0,+\infty)$.
From \eqref{alc}, this implies
\[
\alpha_2\!\left(\alpha_4\!\bigl(\|u\|_{\infty,t_0}\bigr)\right) < a_1.
\]
We claim that, if there exist 
$$c\in \left[ 
\alpha_2\left( \alpha_4\left(\|u\|_{\infty,t_0} \right) \right), a_1 \right)$$
 and $t_1\in[t_0,+\infty)$ such that $x_u(t_1)\in S_{c,t_1}$, then $x_u(t)\in S_{c,t}$ for all $t\in[t_1,+\infty)$.

 We prove the claim by contradiction. Suppose, to the contrary, that there
 exists \(t_2\in(t_1,+\infty)\) such that
 \(x_u(t_2)\in B_{x^*}(r_0)\setminus S_{c,t_2}\), that is,
 \(V(t_2,x_u(t_2))>c\).
 Since \(x_u(\cdot)\) is locally absolutely continuous and \(V\) is of class
 \(C^1\), the function \(V(\cdot,x_u(\cdot))\) is locally absolutely
 continuous. Moreover, for a.e. \(t\geq t_0\),
 \[
 \frac{d}{dt}V(t,x_u(t))
 =
 V_t(t,x_u(t))+d_xV(t,x_u(t))f(t,x_u(t),u(t)).
 \]

 Define
 $
 \Omega:=\{t\in[t_1,t_2]\mid V(t,x_u(t))>c\}$.
 Since \(V(\cdot,x_u(\cdot))\) is continuous, \(\Omega\) is relatively open
 in \([t_1,t_2]\). Moreover, \(t_2\in\Omega\), so
 \(\Omega\neq\emptyset\). Let \(I\) be the connected component of \(\Omega\)
 containing \(t_2\). Since \(V(t_1,x_u(t_1))\leq c\), there exists
 \(a\in[t_1,t_2)\) such that
 $
 I=(a,t_2]$
 in the relative topology of \([t_1,t_2]\). In particular,
 \[
 V(a,x_u(a))\leq c,
 \qquad
 V(t,x_u(t))>c\geq
 \alpha_2\left(\alpha_4\left(\|u\|_{\infty,t_0}\right)\right),
 \quad \forall\,t\in(a,t_2].
 \]
 By (\ref{iss1}) and Step 3,
 \[
 V(t,x_u(t))\leq\alpha_2(\rho(x_u(t),x^*)),
 \qquad t\in[t_0,+\infty).
 \]
 It follows that
 \[
 \rho(x_u(t),x^*)>
 \alpha_4(\|u(\cdot)\|_{\infty,t_0}),
 \qquad t\in(a,t_2].
 \]
 On the other hand, by (\ref{a4u}),
 \[
 |u(t)|\leq\|u(\cdot)\|_{\infty,t_0}<\delta_V
 \quad\text{for a.e. }t\geq t_0.
 \]
 Thus,
 \[
 \rho(x_u(t),x^*)>\alpha_4(|u(t)|)
 \quad\text{for a.e. }t\in(a,t_2).
 \]
 By the ISS-Lyapunov inequality, we obtain
 \[
 \frac{d}{dt}V(t,x_u(t))
 \leq-\alpha_3(\rho(x_u(t),x^*))<0
 \quad\text{for a.e. }t\in(a,t_2).
 \]
 Take any \(s\in(a,t_2)\). Since \(V(\cdot,x_u(\cdot))\) is absolutely
 continuous on \([a,s]\), we have
 \[
 V(s,x_u(s))-V(a,x_u(a))
 =
 \int_a^s\frac{d}{d\tau}V(\tau,x_u(\tau))\,d\tau<0.
 \]
 Therefore,
 \[
 V(s,x_u(s))<V(a,x_u(a))\leq c,
 \]
 which contradicts \(s\in(a,t_2)\subset\Omega\). Hence no such \(t_2\)
 exists, and
 $
 V(t,x_u(t))\leq c$ for all $t\geq t_1$.
 Equivalently,
 $
 x_u(t)\in S_{c,t}$ for all $t\geq t_1$.
 This proves the claim.

 {\it Step 5.}\;Let the input $u(\cdot)\in L^\infty([0,+\infty);\mathbb R^m)$ and the initial time $t_0\in[0,+\infty)$ satisfy (\ref{uhat}). Then (\ref{a4u}) holds. Assuming the initial state satisfies 
$x_u(t_0) \in B_{x^*}(r_V) \setminus \operatorname{cl} B_{x^*}\!\left( \alpha_4\!\left(\|u\|_{\infty,t_0}\right) \right)$, 
we set
\begin{align*}
T_{u,t_0} :=
\inf \left\{ t > 0 \mid
x_u(t_0 + t) \in
\operatorname{cl} B_{x^*}\!\left(
\alpha_4\!\left(\|u\|_{\infty,t_0}\right)
\right)
\right\},
\end{align*}
with the convention that \(T_{u,t_0}=+\infty\) if the above set is empty.
Then, $T_{u,t_0} > 0$.

We claim that
\[
\rho(x^*,x_u(t)) \leq \alpha_1^{-1}\!\Bigl( \beta_0\bigl( \alpha_2(\rho(x^*,x_u(t_0))),\, t - t_0 \bigr) \Bigr),
\qquad 
\forall\, t \in [t_0,\, t_0 + T_{u,t_0}),
\]
where  the function $\beta_0$ is given in Step 1.

Indeed, we have 
$
x_u(t) \in D \setminus \operatorname{cl} B_{x^*}\!\Bigl( \alpha_4\bigl(\|u\|_{\infty,t_0} \bigr) \Bigr)
$
for all $t \in [t_0,\, t_0+T_{u,t_0})$.
If \(\|u\|_{\infty,t_0}>0\), this gives
\[
\rho(x_u(t),x^*)>\alpha_4(\|u\|_{\infty,t_0}).
\]
If \(\|u\|_{\infty,t_0}=0\), the corresponding weak inequality follows from
\(\rho(x_u(t),x^*)\geq0=\alpha_4(0)\). Hence, in both cases,
\[
\rho(x_u(t),x^*)\geq\alpha_4(\|u\|_{\infty,t_0}).
\]
Since $|u(t)|\leq\|u\|_{\infty,t_0}<\delta_V$ and
$\alpha_4(|u(t)|)\leq\alpha_4(\|u\|_{\infty,t_0})$ for a.e.
$t \in [t_0,\, t_0+T_{u,t_0})$, we obtain from \eqref{iss2} that,
for a.e. $t \in [t_0,\, t_0+T_{u,t_0})$,
\begin{align*}
    \frac{d}{dt} V(t,x_u(t))
    &= V_t(t,x_u(t)) 
       + d_x V(t,x_u(t)) \, \left(f(t,x_u(t),u(t))\right)  \\
    &\le -\alpha_3(\rho(x^*,x_u(t)))  \\
    &\le 0.
\end{align*}
Consequently, \eqref{iss1} and \eqref{alc} imply
\begin{align*}
    V(t,x_u(t)) 
    \le V(t_0,x_u(t_0)) 
    \le \alpha_2(\rho(x^*,x_u(t_0))) 
    < a_1, \qquad 
    \forall\, t \in [t_0,\, t_0+T_{u,t_0}).
\end{align*}
Thus, $\widetilde\alpha(V(t,x_u(t)))$ is well defined on this interval.
Moreover, \eqref{iss1}, the monotonicity of $\alpha_2$, and
\eqref{tilde-alpha} yield, for a.e.
$t\in[t_0,t_0+T_{u,t_0})$,
\begin{align*}
    \frac{d}{dt} V(t,x_u(t))
    &\le -\alpha_3(\rho(x^*,x_u(t)))\le -\alpha_3 \circ \alpha_2^{-1}\!\bigl(V(t,x_u(t))\bigr)\le -\widetilde\alpha\!\bigl(V(t,x_u(t))\bigr).
\end{align*}
Then, from Step~1, \cite[Lemma~3.4, p.~102]{jind} and \eqref{iss1}, we obtain for all $t \in [t_0,\, t_0+T_{u,t_0})$ that
\begin{align*}
   \alpha_1(\rho(x^*,x_u(t)))
    &\le V(t,x_u(t))
   \\ &\le \beta_0\!\bigl( V(t_0,x_u(t_0)),\, t-t_0 \bigr)  \\
    &\le \beta_0\!\bigl( \alpha_2(\rho(x^*,x_u(t_0))),\, t-t_0 \bigr),
\end{align*}
which yields
\begin{align*}
    \rho(x^*,x_u(t))
    \le \alpha_1^{-1} \!\Bigl( \beta_0\bigl( \alpha_2(\rho(x^*,x_u(t_0))),\, t-t_0 \bigr) \Bigr),
    \qquad 
    \forall\, t \in [t_0,\, t_0+T_{u,t_0}).
\end{align*}
This completes the proof of the claim.

If \(T_{u,t_0}<+\infty\), then continuity and the definition of
\(T_{u,t_0}\) give
\[
x_u(t_0+T_{u,t_0})
\in
\operatorname{cl}B_{x^*}\!\left(\alpha_4(\|u\|_{\infty,t_0})\right).
\]
Set
\[
c_u:=\alpha_2\!\left(\alpha_4(\|u\|_{\infty,t_0})\right).
\]
By \eqref{alc} and \eqref{bais},
\(x_u(t_0+T_{u,t_0})\in S_{c_u,t_0+T_{u,t_0}}\). Hence Step~4 yields
\[
V(t,x_u(t))\leq c_u,\qquad \forall\,t\geq t_0+T_{u,t_0}.
\]
Together with \eqref{iss1}, this gives
\[
\rho(x_u(t),x^*)
\leq
(\alpha_1^{-1}\circ\alpha_2\circ\alpha_4)(\|u\|_{\infty,t_0}),
\qquad \forall\,t\geq t_0+T_{u,t_0}.
\]

{\it Step 6.}\;Let the input $u(\cdot) \in L^\infty([0,+\infty);\mathbb R^m)$ and the initial time $t_0 \ge 0$ satisfy 
$0 < \|u\|_{\infty,t_0}<\hat{r}_V$. Assume that
\[
x_u(t_0) \in \operatorname{cl} B_{x^*}\!\left( \alpha_4\!\left( \|u\|_{\infty,t_0}\right) \right).
\]
With \(c_u:=\alpha_2\!\left(\alpha_4(\|u\|_{\infty,t_0})\right)\), \eqref{alc} and
\eqref{bais} imply \(x_u(t_0)\in S_{c_u,t_0}\). Applying Step~4 and then
\eqref{iss1}, we obtain
\[
\rho(x^*,x_u(t)) \leq \alpha_1^{-1}\!\left( V(t,x_u(t)) \right)
\leq (\alpha_1^{-1} \circ \alpha_2 \circ \alpha_4)\!\left( \|u\|_{\infty,t_0} \right),
\qquad 
\forall\, t \in [t_0, +\infty).
\]
This establishes the required estimate when
\(x_u(t_0)\in
\operatorname{cl} B_{x^*}
\bigl(\alpha_4(\|u\|_{\infty,t_0})\bigr)\).

{\it Step 7.}\;
To conclude the proof, 
we set
\begin{align*}
&\beta(s,t)=\alpha_1^{-1}\!\left(\beta_0\!\left(\alpha_2(s),t\right)\right),\quad\forall\,(s,t)\in[0,r_V)\times[0,+\infty),\\
&\gamma(s)=\alpha_1^{-1}\!\left(\alpha_2\!\left(\alpha_4(s)\right)\right),\quad\forall\,s\in[0,\hat{r}_V).
\end{align*}
For any $s\in[0,\hat{r}_V)$, by the definition of $\hat{r}_V$, we have
$s<\delta_V$ and $\alpha_4(s)<r_V$. Hence, by \eqref{alc},
$\alpha_2(\alpha_4(s))<a_1$, and $\gamma$ is well defined on
$[0,\hat{r}_V)$. It follows from \cite[Lemma~4.2, p.~145]{jind} that
$\beta$ is a $\mathcal{KL}$-function and $\gamma$ is a
$\mathcal K$-function.

We obtain from Step 3 that, for any $u(\cdot)\in L^\infty([0,+\infty);\mathbb R^m)$ and $t_0\in[0,+\infty)$ satisfying $\|u\|_{\infty,t_0}<\hat{r}_V$, if the initial state $x_u(t_0)\in B_{x^*}(r_V)$, every corresponding maximal solution $x_u(\cdot)$ is defined on $[t_0,+\infty)$. Moreover, it follows from Step 5 and Step 6 that inequality \eqref{isse} holds. This completes the proof.
\end{proof}

We next prove Corollary~\ref{global-ISS}.

\begin{proof}[Proof of Corollary \ref{global-ISS}]
Since \(D=M\), we have \(r_0=+\infty\). Because
\(\alpha_1,\alpha_2\) are \(\mathcal K_\infty\)-functions,
\(\alpha_2^{-1}\circ\alpha_1\) is also a
\(\mathcal K_\infty\)-function. Hence
\[
 r_V
 =
 \lim_{r\to+\infty}
 \alpha_2^{-1}\circ\alpha_1(r)
 =
 +\infty.
\]
Moreover, \(\delta_V=+\infty\) and
\(\alpha_4(s)<+\infty=r_V\) for every \(s\geq0\).
Hence
\[
 \widehat r_V
 =
 \sup\{s\geq0:\alpha_4(s)<r_V\}
 =
 +\infty.
\]
Therefore, Theorem~\ref{isslM} gives a
\(\mathcal{KL}\)-function \(\beta\) defined on
\([0,+\infty)\times[0,+\infty)\) and a \(\mathcal K\)-function
\(\gamma\) defined on \([0,+\infty)\), and the estimate
\eqref{isse} holds for every
\[
u(\cdot)\in L^\infty([0,+\infty);\mathbb R^m).
\]
Thus,
system~\eqref{sm} is globally ISS.
\end{proof}

We now prove Theorem \ref{isse31}.

\begin{proof}[Proof of Theorem \ref{isse31}]

Let \(V\) be the ISS-Lyapunov function in Theorem~\ref{isslM}, with
\(\mathcal K\)-functions
\(\alpha_1,\allowbreak \alpha_2,\allowbreak \alpha_3,\allowbreak \alpha_4\) as in
\eqref{iss1} and \eqref{iss2}. Set
\begin{align*}
   & \widehat V(t,x)=
   \frac{V(t,x)}{r_0^2-\rho^2(x,x^*)},
   \quad
   \forall\,(t,x)\in[0,+\infty)\times B_{x^*}(r_0),\\
   & \widehat\alpha_i(r)=
   \frac{\alpha_i(r)}{r_0^2-r^2},
   \quad
   \forall\,r\in[0,r_0),\quad i=1,2,3,\\
   & \widehat\delta_V=
   \sup\left\{s\in[0,\delta_V)\mid \alpha_5(s)<r_0\right\},\\
   & \widehat\alpha_4(s)=\max\{\alpha_4(s),\alpha_5(s)\},
   \quad
   \forall\,s\in[0,\widehat\delta_V).
\end{align*}
Since \(r_0\leq i(x^*)\), the function \(x\mapsto \rho^2(x,x^*)\) is of
class \(C^1\) on \(B_{x^*}(r_0)\). Hence, \(\widehat V\) is of class
\(C^1\) on \([0,+\infty)\times B_{x^*}(r_0)\).

It is clear that \(\widehat\alpha_i\), \(i=1,2,3\), are
\(\mathcal K\)-functions. Since \(\alpha_5(0)=0<r_0\), we have \(\widehat\delta_V>0\), and hence \(\widehat\alpha_4\) is also a
\(\mathcal K\)-function on \([0,\widehat\delta_V)\). Moreover, by \eqref{ha1},
\[
    \lim_{r\to r_0^-}\widehat\alpha_1(r)=+\infty.
\]

From \eqref{iss1}, we have
\begin{align*}
    \widehat\alpha_1(\rho(x,x^*))
    \leq \widehat V(t,x)
    \leq \widehat\alpha_2(\rho(x,x^*)),
    \quad
    \forall\,(t,x)\in[0,+\infty)\times B_{x^*}(r_0).
\end{align*}

Next, take any
\[
(t,x,u)\in[0,+\infty)\times B_{x^*}(r_0)\times\mathbb R^m
\]
such that
\[
    |u|<\widehat\delta_V,
    \qquad
    \rho(x,x^*)\geq \widehat\alpha_4(|u|).
\]
Then
\[
    |u|<\delta_V,
    \qquad
    \rho(x,x^*)\geq \alpha_4(|u|)
    \quad\text{and}\quad
    \rho(x,x^*)\geq \alpha_5(|u|).
\]
Therefore, by \eqref{iss2} and \eqref{drf3}, we obtain
\begin{align*}
    &\widehat V_t(t,x)
    +d_x\widehat V(t,x)\bigl(f(t,x,u)\bigr)\\
    &=
    \frac{
    V_t(t,x)+d_xV(t,x)\bigl(f(t,x,u)\bigr)
    }{
    r_0^2-\rho^2(x,x^*)
    }\\
    &\quad
    +
    \frac{
    V(t,x)
    }{
    \bigl(r_0^2-\rho^2(x,x^*)\bigr)^2
    }
    d_x\rho^2(x,x^*)\bigl(f(t,x,u)\bigr)\\
    &\leq
    -\frac{\alpha_3(\rho(x,x^*))}
    {r_0^2-\rho^2(x,x^*)}\\
    &=
    -\widehat\alpha_3(\rho(x,x^*)).
\end{align*}
Thus, \(\widehat V\) is an ISS-Lyapunov function on
\(B_{x^*}(r_0)\).

Since
\[
    \lim_{r\to r_0^-}\widehat\alpha_1(r)=+\infty
    \quad\text{and}\quad
    \lim_{r\to r_0^-}\widehat\alpha_2(r)=+\infty,
\]
we have
\[
    \lim_{r\to r_0^-}
    \widehat\alpha_2^{-1}\circ\widehat\alpha_1(r)=r_0.
\]
Hence, applying Theorem~\ref{isslM} to the ISS-Lyapunov function
\(\widehat V\), we conclude that system \eqref{sm} is regionally ISS in the set
\(B_{x^*}(r_0)\). This completes the proof.
\end{proof}

\section{Appendix}\label{aur9}

For completeness, we collect some notions that are used to state the problem and the main results.

\begin{definition}\label{kkl1}
Let $r>0$. A function $\alpha:[0,r)\to[0,+\infty)$ is said to be a
\emph{\(\mathcal K\)-function} if it is continuous, strictly increasing, and
satisfies $\alpha(0)=0$.

If $r=+\infty$ and
$
\lim_{s\to+\infty}\alpha(s)=+\infty$, 
then $\alpha$ is said to be a \emph{\(\mathcal K_\infty\)-function}.

A continuous function $\beta:[0,r)\times[0,+\infty)\to[0,+\infty)$ is said
to be a \emph{\(\mathcal{KL}\)-function} if, for each fixed $t\geq0$, the
function $\beta(\cdot,t)$ is a $\mathcal K$-function, and for each fixed
$s\in[0,r)$, the function $\beta(s,\cdot)$ is decreasing and satisfies
$
\lim_{t\to+\infty}\beta(s,t)=0.
$
\end{definition}

We first recall the chain rule for differentials of \(C^1\) maps between
smooth manifolds; see \cite[Proposition~8.13, p.~256]{gallier-quaintance}.

\begin{lemma}\label{cr}
    Let $M_1$, $M_2$, and $M_3$ be smooth manifolds with or without boundary,
    let $F_1: M_1 \to M_2$ and $F_2: M_2 \to M_3$ be \(C^1\) maps, and let
    $x\in M_1$. Then the following chain rule holds:
    \begin{align*}
        d(F_2 \circ F_1)_x = d(F_2)_{F_1(x)} \circ d(F_1)_x \colon T_x M_1 \to T_{F_2 \circ F_1(x)} M_3.
    \end{align*}
    Moreover, if $F_1$ is a \(C^1\)-diffeomorphism, then $d(F_1)_x : T_x M_1 \to T_{F_1(x)} M_2$ is an isomorphism, and $\left(d(F_1)_x\right)^{-1} = d(F_1^{-1})_{F_1(x)}$.
\end{lemma}

The last assertion follows by applying the chain rule to
\(F_1^{-1}\circ F_1=\operatorname{id}_{M_1}\) and
\(F_1\circ F_1^{-1}=\operatorname{id}_{M_2}\).

Next, we collect the notation and identities for covariant derivatives,
cotangent norms, and the squared distance function that will be used below.
We keep the notation \(B_x(r)\) and \(B(O_x,r)\) from
\eqref{metric-ball} and \eqref{tangent-space-ball}, respectively, and the
notation \(O_x\), \(\exp_x\), and \(i(x)\) introduced before Theorem
\ref{isse31}. By \cite[Theorem~5.10, p.~122]{lee8}, there exists a unique
connection on \(TM\) that is compatible with \(g\) and symmetric. It is called
the \emph{Levi--Civita connection} of \(g\). Throughout this appendix,
\(\nabla\) denotes this connection and the connections induced by it on the
tensor bundles of \(M\). If \(\sigma:I\to M\) is a smooth curve and \(V\) is a
smooth vector field along \(\sigma\), then
\(\nabla_{\dot\sigma(t)}V(t)\) denotes the covariant derivative of \(V\) along
\(\sigma\) at \(t\); this derivative is denoted by \(D_tV(t)\) in
\cite[Theorem~4.24, pp.~101--102]{lee8}. For a smooth function
\(h:M\to\mathbb R\), \(\nabla h\) denotes its covariant derivative, regarded
as a cotangent vector field. It follows from \cite[(5.9)]{dzgJ} and
\cite[pp.~281 and~284]{le1} that, for any \(x\in M\),
\begin{align}\label{dhnh}
   dh_x(X)=\nabla h(x)(X)=X(h),\quad \forall X\in T_xM. 
\end{align}
If \(\eta\in T_x^*M\), its norm is
\begin{align}\label{cotangent-norm}
    |\eta|:=\sup\{|\eta(X)|:X\in T_xM,\ |X|\leq1\}.
\end{align}
For \(v\in T_xM\), its dual covector \(v^\top\in T_x^*M\) is defined by
\begin{align}\label{dual-covector}
 v^\top(X)=g_x(v,X),\quad \forall X\in T_xM.
\end{align}
Then we have
\begin{align}\label{ndce1}
   |v^\top|=|v|.
\end{align}
Indeed, the Cauchy--Schwarz inequality gives \(|v^\top|\leq|v|\), while
equality follows by taking \(X=v/|v|\) when \(v\ne O_x\).

The definitions \eqref{cotangent-norm} and \eqref{dual-covector} are the
corresponding special cases of those in \cite[Section~2.2]{cdz} and
\cite[Sections~2.1 and~5.2]{dzgJ}.

We next prove Lemma~\ref{lem:pointwise-chain-rule}.

\begin{proof}[Proof of Lemma~\ref{lem:pointwise-chain-rule}]
Set \(y=\varphi\circ\gamma\). Since \(y(\cdot)\) is differentiable at \(t_0\),
it is continuous there, and hence so is \(\gamma(\cdot)\). Therefore,
\(\gamma(t)\in V\) for \(t\) sufficiently close to \(t_0\). By restricting
\(\varphi\) to \(U\cap V\) and shrinking \(J\), we may assume that
\(U\subset V\). If \(\psi\) is another coordinate chart about \(x\), then, for
\(t\) sufficiently close to \(t_0\),
\[
 \psi\circ\gamma=(\psi\circ\varphi^{-1})\circ y.
\]
Since \(\psi\circ\varphi^{-1}\) is a smooth map between open subsets of
Euclidean spaces, the ordinary Euclidean chain rule at \(t_0\) gives
\[
 (\psi\circ\gamma)'(t_0)
 =d(\psi\circ\varphi^{-1})_{\varphi(x)}\big(y'(t_0)\big).
\]
This pointwise form of the Euclidean chain rule requires only that \(y\) be
differentiable at \(t_0\) and that the outer map be differentiable at
\(y(t_0)=\varphi(x)\); it does not require \(y\) to be of class \(C^1\) in a
neighbourhood of \(t_0\).
Here \(d(\psi\circ\varphi^{-1})_{\varphi(x)}\) denotes the usual differential
of a map between Euclidean spaces. By the chain rule for differentials in
Lemma \ref{cr}, if
\[
 v=d(\varphi^{-1})_{\varphi(x)}\big(y'(t_0)\big),
\]
then
\[
 (\psi\circ\gamma)'(t_0)=d\psi_x(v)
 \quad\text{and}\quad
 d(\psi^{-1})_{\psi(x)}\big((\psi\circ\gamma)'(t_0)\big)=v,
\]
which proves that the vector in
\eqref{pointwise-velocity} is independent of the chosen chart.

If \(\gamma\) is of class \(C^1\), then the chain rule gives
\[
(\varphi\circ\gamma)'(t_0)
=
d\varphi_x\left(
d\gamma_{t_0}
\left(\left.\frac{d}{dt}\right|_{t_0}\right)
\right).
\]
Therefore, Lemma~\ref{cr} yields
\[
d(\varphi^{-1})_{\varphi(x)}
\bigl((\varphi\circ\gamma)'(t_0)\bigr)
=
d\gamma_{t_0}
\left(\left.\frac{d}{dt}\right|_{t_0}\right),
\]
where the right-hand side is the usual velocity vector of \(\gamma\) at
\(t_0\) (see \eqref{dgtzs}).

Since \(h=(h\circ\varphi^{-1})\circ\varphi\) on \(U\), Lemma \ref{cr}
gives the coordinate representation
\begin{align}\label{coordinate-differential-h}
 dh_x=d(h\circ\varphi^{-1})_{\varphi(x)}\circ d\varphi_x.
\end{align}
Since \(h\circ\varphi^{-1}\) is of class \(C^1\), the same pointwise Euclidean
chain rule yields
\begin{align*}
    \left.\frac{d}{dt}\right|_{t=t_0}h(\gamma(t))
 &=\left.\frac{d}{dt}\right|_{t=t_0}
 (h\circ\varphi^{-1})\big((\varphi\circ\gamma)(t)\big)
 \\&=d(h\circ\varphi^{-1})_{\varphi(x)}\big((\varphi\circ\gamma)'(t_0)\big)
 \\&=dh_x\big(\dot\gamma(t_0)\big),
\end{align*}
where the last equality follows from \eqref{coordinate-differential-h},
\eqref{pointwise-velocity} and Lemma \ref{cr}. This proves \eqref{pointwise-chain-rule}.
\end{proof}

We next derive a pointwise chain rule for the first partial covariant
derivative of a smooth function on \(M\times M\). This identity will be used
in the proof of Lemma~\ref{lem:distance-ac-curve}.
Let \(\Omega\subset M\times M\) be an open set, and let
\(F:\Omega\to\mathbb R\) be a smooth function. For \(x,y\in M\), set
\[
 \Omega_x:=\{q\in M:(x,q)\in\Omega\},
 \qquad
 \Omega^y:=\{q\in M:(q,y)\in\Omega\}.
\]
The sets \(\Omega_x\) and \(\Omega^y\) are open. At each
\((x,y)\in\Omega\), \(\nabla_1F(x,y)\in T_x^*M\) and
\(\nabla_2F(x,y)\in T_y^*M\) denote the covariant derivatives with respect
to the first and second arguments, respectively. Thus,
\(\nabla_1F(\cdot,y)\) is the covariant derivative of the smooth function
\(F(\cdot,y):\Omega^y\to\mathbb R\). If \((x,y)\in\Omega\),
\(X\in T_xM\), and \(Y\in T_yM\), we define
\begin{align}\label{mixed-first-covector}
 \nabla_2\nabla_1F(x,y)(X,\cdot)
 :=\nabla\big(q\mapsto\nabla_1F(x,q)(X)\big)(y)\in T_y^*M,
\end{align}
where \(\nabla\big(q\mapsto\nabla_1F(x,q)(X)\big)(y)\) denotes the covariant
derivative at \(y\) of the smooth function
\(q\mapsto\nabla_1F(x,q)(X)\) on \(\Omega_x\). We further define
\begin{align}\label{mixed-pairing}
 \nabla_2\nabla_1F(x,y)(X,Y)
 :=\nabla_2\nabla_1F(x,y)(X,\cdot)(Y),
\end{align}
where the right-hand side denotes the pairing of the cotangent vector
\(\nabla_2\nabla_1F(x,y)(X,\cdot)\in T_y^*M\) with the tangent vector
\(Y\in T_yM\). These definitions are the corresponding special cases of those
in \cite[Section~2.2]{cdz} and \cite[Section~5.2]{dzgJ}.

We now apply Lemma \ref{lem:pointwise-chain-rule} to derive a pointwise
chain rule for \(\nabla_1F\) along a curve. Let \(\Omega\) and \(F\) be as above, let
\(\gamma\) satisfy the assumptions of that lemma, and set \(y=\gamma(t_0)\).
Fix \(x\in M\) such that \((x,y)\in\Omega\). Since \(\Omega_x\) is an open
neighborhood of \(y\) and \(\gamma\) is continuous at \(t_0\), we have
\(\gamma(t)\in\Omega_x\) for all \(t\) sufficiently close to \(t_0\).
Therefore, for every \(X\in T_xM\), we may apply Lemma
\ref{lem:pointwise-chain-rule} on the open submanifold \(\Omega_x\) to the
smooth function \(q\mapsto\nabla_1F(x,q)(X)\). Then \eqref{dhnh},
\eqref{pointwise-chain-rule}, and
\eqref{mixed-first-covector}--\eqref{mixed-pairing} give
\begin{align}\label{mixed-chain-rule-scalar}
 \left.\frac{d}{dt}\right|_{t=t_0}
 \nabla_1F(x,\gamma(t))(X)
 =\nabla_2\nabla_1F(x,y)\big(X,\dot\gamma(t_0)\big).
\end{align}
Since \(T_x^*M\) is finite-dimensional, applying
\eqref{mixed-chain-rule-scalar} to the elements of a basis of \(T_xM\) yields
the following identity in \(T_x^*M\):
\begin{align}\label{mixed-chain-rule-covector}
 \lim_{s\to0}
 \frac{\nabla_1F(x,\gamma(t_0+s))-\nabla_1F(x,\gamma(t_0))}{s}
 =
 \nabla_2\nabla_1F(x,y)(\,\cdot\,,\dot\gamma(t_0)).
\end{align}

The following lemma records the identities for the squared distance function
that are needed below; see \cite[Lemmas~2.1--2.2]{cdz} and
\cite[Lemmas~5.1--5.2]{dzgJ}.
\begin{lemma}\label{prms}
For any \(x,y\in M\) satisfying
\(\rho(x,y)<\min\{i(x),i(y)\}\), the function \(\rho^2\) is smooth in a
neighbourhood of \((x,y)\), and
\begin{align}
 |\exp_x^{-1}y|&=|\exp_y^{-1}x|=\rho(x,y),\label{exp-distance}\\
 \nabla_1\rho^2(x,y)&=-2(\exp_x^{-1}y)^\top,\qquad
 \nabla_2\rho^2(x,y)=-2(\exp_y^{-1}x)^\top.\label{rho2-gradient}
\end{align}
Consequently, by \eqref{exp-distance}, \eqref{rho2-gradient}, and
\eqref{ndce1},
\begin{align}\label{rho-gradient-norm}
 \rho(x,y)=\frac12|\nabla_1\rho^2(x,y)|.
\end{align}
In particular, for any \(x\in M\),
\begin{align}\label{rho2-diagonal-gradient}
 \nabla_1\rho^2(x,x)=0_x^*.
\end{align}
Here \(0_x^*\) denotes the zero covector in \(T_x^*M\).
Moreover, for any \(x\in M\) and \(Y,Z\in T_xM\),
\begin{align}\label{rho2-diagonal-mixed}
 \nabla_2\nabla_1\rho^2(x,x)(Y,Z)=-2g_x(Y,Z).
\end{align}
Equivalently, by \eqref{dual-covector} and \eqref{ndce1}, for each fixed \(Z\),
\begin{align}\label{rho2-diagonal-mixed-covector}
 \nabla_2\nabla_1\rho^2(x,x)(\,\cdot\,,Z)=-2Z^\top,\qquad
 \left|\nabla_2\nabla_1\rho^2(x,x)(\,\cdot\,,Z)\right|=2|Z|.
\end{align}
\end{lemma}

Identity \eqref{rho2-diagonal-mixed} is obtained by taking \(x=y\) in
\cite[(2.20)]{cdz} and then using \cite[Lemma~2.1]{cdz} to identify the
differential of \(\exp_x^{-1}\) at \(x\) under the canonical identification.
We now use the preceding identities to relate the metric derivative, defined
in \eqref{mdac1} below, of a locally absolutely continuous curve to the norm of its
velocity vector.
\begin{lemma}\label{lem:distance-ac-curve}
Let \((M,g)\) be a connected and complete Riemannian manifold, let $c_1,c_2\in\mathbb R$ with \(c_1<c_2\), and let \(x:(c_1,c_2)\to M\) be locally absolutely
continuous. Then \(\dot x(\tau)\), defined by \eqref{pointwise-velocity}, exists
for a.e. \(\tau\in(c_1,c_2)\),
\(|\dot x(\cdot)|\in L^1_{\rm loc}(c_1,c_2)\), and
\[
    \rho(x(s),x(t))
    \leq
    \int_s^t |\dot x(\tau)|\,d\tau,\qquad c_1<s\leq t<c_2.
\]
\end{lemma}

\begin{proof}
Fix a compact interval \([a,b]\subset(c_1,c_2)\). By Definition
\ref{jDlx}, the restriction of \(x(\cdot)\) to \([a,b]\) is absolutely continuous.
Recall that, for a curve in a metric space, the metric derivative at
\(\tau\) is defined by
\begin{align}\label{mdac1}
   |x'|_\rho(\tau)
    :=\lim_{h\to0}
    \frac{\rho(x(\tau+h),x(\tau))}{|h|} 
\end{align}
whenever this limit exists. It follows from
\cite[Theorem~1.1.2, p.~24]{MR2129498} that this metric derivative
exists for a.e. \(\tau\in(a,b)\), belongs to \(L^1(a,b)\), and satisfies
\[
    \rho(x(s),x(t))
    \leq \int_s^t |x'|_\rho(\tau)\,d\tau,
    \qquad a\leq s\leq t\leq b.
\]

We next identify the metric derivative with the norm of the velocity vector.
We use the standard notions of geodesics and geodesic segments from
\cite[pp.~103--104]{lee8}. If
\(\sigma:[\alpha,\beta]\to M\) is piecewise differentiable, its Riemannian
length is
\begin{align}\label{riemannian-curve-length}
 L_g(\sigma):=\int_\alpha^\beta|\dot\sigma(r)|\,dr;
\end{align}
see \cite[p.~34]{lee8}. A curve from \(p\) to \(q\) is \emph{minimizing} if its
length equals \(\rho(p,q)\); see \cite[p.~151]{lee8}. Recall also that a set
\(U\subset M\) is \emph{geodesically convex} if, for every \(p,q\in U\), there is a
unique minimizing geodesic segment from \(p\) to \(q\) whose image is contained
in \(U\); see \cite[p.~166]{lee8}.

The curve \(x(\cdot)\) is continuous by Definition \ref{jDlx}. Fix
\(\tau_0\in[a,b]\), and choose a relatively compact coordinate neighborhood
\(W\) of \(x(\tau_0)\). By \cite[Theorem~6.17, p.~166]{lee8}, after reducing
the radius if necessary, we may choose a geodesically convex geodesic ball
\(U\) centred at \(x(\tau_0)\) such that \(\overline U\subset W\). Let
\(\varphi:W\to\mathbb R^n\) be the corresponding coordinate map. Since
\(\overline U\) is compact and \(d\varphi\) is continuous, there is a constant
\(C>0\) such that
\[
 |d\varphi_z(v)|\leq C|v|,
 \qquad z\in U,\quad v\in T_zM.
\]
Take any distinct \(p,q\in U\). By the definition of geodesic convexity, there
is a unique minimizing geodesic segment
\(\sigma:[\alpha,\beta]\to M\) from \(p\) to \(q\), and its image is contained
in \(U\). Thus, \(\sigma(\alpha)=p\), \(\sigma(\beta)=q\), and
\(L_g(\sigma)=\rho(p,q)\). By Lemma \ref{cr},
\[
 \frac{d}{dr}\varphi(\sigma(r))
 =d\varphi_{\sigma(r)}(\dot\sigma(r)),\qquad r\in[\alpha,\beta].
\]
Therefore, the Newton--Leibniz formula in \(\mathbb R^n\) and the standard
estimate \(\left|\int_\alpha^\beta v(r)\,dr\right|
\leq\int_\alpha^\beta|v(r)|\,dr\) for continuous functions
\(v:[\alpha,\beta]\to\mathbb R^n\) yield
\begin{align*}
 |\varphi(p)-\varphi(q)|
 &=\left|\int_\alpha^\beta
       d\varphi_{\sigma(r)}(\dot\sigma(r))\,dr\right|\\
 &\leq \int_\alpha^\beta
       |d\varphi_{\sigma(r)}(\dot\sigma(r))|\,dr
 \leq C\int_\alpha^\beta|\dot\sigma(r)|\,dr
 =C L_g(\sigma)=C\rho(p,q).
\end{align*}
The same inequality is immediate when \(p=q\).
By continuity, there is an interval \(I\) containing \(\tau_0\) such
that \(x(I\cap[a,b])\subset U\). Let \(m\in L^1(a,b)\) be given by Definition
\ref{jDlx}. For \(s,t\in I\cap[a,b]\) with \(s\leq t\),
\[
    |\varphi(x(t))-\varphi(x(s))|
    \leq C\rho(x(s),x(t)) \leq C\int_s^t m(\tau)\,d\tau.
\]
Thus, the coordinate representation \(\varphi\circ x(\cdot)\) is absolutely
continuous on \(I\cap[a,b]\), and therefore it is
differentiable for a.e. \(\tau\in I\cap(a,b)\). By the coordinate representation of
velocity vectors (\ref{pointwise-velocity}), \(\dot x(\tau)\) is therefore defined for
a.e. \(\tau\in I\cap(a,b)\). Since \([a,b]\) can be covered by finitely many such
intervals, \(\dot x(\tau)\) exists for a.e. \(\tau\in(a,b)\).

Fix such a \(\tau\) at which both \(|x'|_\rho(\tau)\) and \(\dot x(\tau)\)
exist, and set \(p=x(\tau)\). By Lemma \ref{prms}, there is an open
neighbourhood \(\Omega\subset M\times M\) of \((p,p)\) such that
\(\rho^2|_\Omega\) is smooth. The slice
\(\Omega_p=\{q\in M:(p,q)\in\Omega\}\) is an open neighbourhood of \(p\).
By the uniformly normal neighborhood lemma
\cite[Lemma~6.14, pp.~163--164]{lee8}, there exist \(\eta>0\) and a
neighbourhood \(N\) of \(p\) that is contained in a geodesic ball of radius
\(\eta\) centred at each of its points. By the definition of the injectivity
radius and \cite[Corollary~6.12, p.~162]{lee8}, \(\eta\leq i(y)\) for every
\(y\in N\) and \(\rho(y,q)<\eta\) for all \(y,q\in N\). Hence, for every
\(q\in N\),
\begin{equation}\label{uniform-normal-injectivity}
 \rho(p,q)<\min\{i(p),i(q)\}.
\end{equation}
In the following calculation, the partial covariant derivatives of
\(\rho^2|_\Omega\) are denoted simply by the corresponding derivatives of
\(\rho^2\).
Since \(x(\tau+h)\to p\) as \(h\to0\), there exists \(\varepsilon>0\) such
that
\[
 x(\tau+h)\in\Omega_p\cap N,\qquad 0<|h|<\varepsilon.
\]
Therefore, \eqref{uniform-normal-injectivity} shows that the hypotheses of
Lemma \ref{prms} are satisfied by \((p,x(\tau+h))\).
By \eqref{rho-gradient-norm} and \eqref{rho2-diagonal-gradient},
\(\rho(p,x(\tau+h))=\frac12|\nabla_1\rho^2(p,x(\tau+h))|\) and
\(\nabla_1\rho^2(p,p)=0_p^*\).
Applying \eqref{mixed-chain-rule-covector} with
\(F=\rho^2|_\Omega\), \(\gamma(\cdot)=x(\cdot)\), and \(t_0=\tau\), we obtain
\begin{align}\label{rho2-along-ac-curve}
 \lim_{h\to0}
 \frac{\nabla_1\rho^2(p,x(\tau+h))-\nabla_1\rho^2(p,p)}{h}
 =
 \nabla_2\nabla_1\rho^2(p,p)(\,\cdot\,,\dot x(\tau))
\end{align}
in \(T_p^*M\). Since the norm on \(T_p^*M\) defined by
\eqref{cotangent-norm} is continuous, \eqref{rho2-along-ac-curve} and
\eqref{rho2-diagonal-mixed-covector} yield
\begin{align*}
    |x'|_\rho(\tau)
    &=\frac12\lim_{h\to0}\left|
      \frac{\nabla_1\rho^2(p,x(\tau+h))-\nabla_1\rho^2(p,p)}{h}\right|
    \\&=\frac12\left|\nabla_2\nabla_1\rho^2
      (p,p)(\,\cdot\,,\dot x(\tau))\right|\\
    &=|\dot x(\tau)|.
\end{align*}
Thus, \(|x'|_\rho(\tau)=|\dot x(\tau)|\) for a.e. \(\tau\in(a,b)\).
On the null set where \(\dot x(\tau)\) is not defined, set
\(\dot x(\tau)=O_{x(\tau)}\). Since \(|x'|_\rho\in L^1(a,b)\), the preceding
equality implies that \(|\dot x(\cdot)|\in L^1(a,b)\).
Substitution into the metric integral inequality proves the desired estimate
on \([a,b]\). Choose \(j_0\in\mathbb N\) such that
\(2/j_0<c_2-c_1\), and, for \(j\geq j_0\), set
\[
 K_j:=\left[c_1+\frac1j,c_2-\frac1j\right].
\]
Then \((K_j)_{j\geq j_0}\) is an increasing sequence of compact intervals
whose union is \((c_1,c_2)\). Applying the preceding argument to each \(K_j\)
and taking the countable union of the corresponding null sets show that
\(\dot x(\tau)\) exists for a.e. \(\tau\in(c_1,c_2)\) and that
\(|\dot x(\cdot)|\in L^1_{\rm loc}(c_1,c_2)\). Finally, every compact interval
\([s,t]\subset(c_1,c_2)\) is contained in some \(K_j\), so the desired estimate
holds for every \(c_1<s\leq t<c_2\).
\end{proof}

The following result concerns the local existence and continuation of
solutions to non-autonomous systems on Riemannian manifolds and extends the
classical Carath\'eodory results for ordinary differential equations in
Euclidean spaces; see \cite[Theorems~5.1--5.2, pp.~28--29]{MR587488}. Its
proof is based on the Euclidean version of the Carath\'eodory existence theorem
and the local coordinate representation of a Riemannian manifold.

\begin{theorem}\label{cyfx1}
Let \((M,g)\) be a connected and complete Riemannian manifold, and let \(D\subset M\) be
an open subset. Let
\[
    b:\mathbb R\times D\to TM
\]
be a time-dependent vector field, that is, \(b(t,x)\in T_xM\) for every
\((t,x)\in \mathbb R\times D\). Assume that \(b\) satisfies the following
Carathéodory-type conditions:
\begin{enumerate}
    \item for every \(x\in D\), the map \(t\mapsto b(t,x)\) is measurable;

    \item for a.e. \(t\in\mathbb R\), the map \(x\mapsto b(t,x)\) is
    continuous;

    \item for every compact set \(J\subset \mathbb R\times D\), there exists
    a function \(I_J\in L^1_{\mathrm{loc}}(\mathbb R)\) such that
    \[
        |b(t,x)|_g\leq I_J(t)
    \]
    for a.e. \(t\in\mathbb R\) and all \(x\in D\) with \((t,x)\in J\).
\end{enumerate}

Consider the initial value problem
\begin{equation}\label{dxb}
\begin{cases}
    \dot x(t)=b(t,x(t)), & \text{for a.e. } t,\\
    x(t_0)=x_0,
\end{cases}
\end{equation}
where \(t_0\in\mathbb R\) and \(x_0\in D\).

Let \(I\subset\mathbb R\) be an interval containing \(t_0\). A map
\(x:I\to D\) is called a solution of \eqref{dxb} on \(I\) if \(x(\cdot)\) is
locally absolutely continuous, \(x(t_0)=x_0\), and
\[
    \dot x(t)=b(t,x(t)),\quad \text{for a.e. } t\in I.
\]

Then the following assertions hold:
\begin{description}
    \item[\textup{(i)}] There exists \(\delta>0\) such that \eqref{dxb}
admits at least one solution on \((t_0-\delta,t_0+\delta)\).

    \item[\textup{(ii)}] Let \(c_2\in(t_0,+\infty]\), and let
    \(x:[t_0,c_2)\to D\) be a solution of \eqref{dxb}. Suppose that
    \(c_2<+\infty\) and that \(x(\cdot)\) cannot be extended beyond \(c_2\), in
    the sense that there do not exist \(c_2'>c_2\) and a solution
    \(\widetilde x:[t_0,c_2')\to D\) of \eqref{dxb} which coincides with
    \(x(\cdot)\) on \([t_0,c_2)\). Then
    \(x(\cdot)\) cannot remain in any compact subset of \(D\) as
    \(t\to c_2^{-}\). More precisely, for
    every compact set \(K\subset D\), there exists \(\tau_K\in [t_0,c_2)\)
    such that
    \[
        x(t)\notin K,
        \qquad
        \forall\, t\in(\tau_K,c_2).
    \]
    Similarly, let \(c_1\in[-\infty,t_0)\), and let
    \(x:(c_1,t_0]\to D\) be a solution of \eqref{dxb}. Suppose that
    \(c_1>-\infty\) and that \(x(\cdot)\) cannot be extended beyond \(c_1\) to
    the left, in the sense that there do not exist \(c_1'<c_1\) and a solution
    \(\widetilde x:(c_1',t_0]\to D\) of \eqref{dxb} which coincides with
    \(x(\cdot)\) on \((c_1,t_0]\). Then
    \(x(\cdot)\) cannot remain in any compact subset of \(D\) as
    \(t\to c_1^{+}\). More precisely, for every compact set \(K\subset D\),
    there exists
    \(\widetilde{\tau}_K\in(c_1,t_0]\) such that
    \[
        x(t)\notin K,
        \qquad
        \forall\, t\in(c_1,\widetilde{\tau}_K).
    \]
\end{description}
\end{theorem}

\begin{remark}\label{rem:caratheodory_continuation}
Theorem~\ref{cyfx1} gives a continuation alternative for
time-dependent tangent vector fields on an open subset of a Riemannian
manifold. If a solution cannot be extended beyond a finite endpoint, then it
eventually leaves every compact subset of the state domain. No local Lipschitz
condition with respect to the state variable,
and hence no uniqueness of solutions, is required.

Related manifold treatments address different settings. Furi and
Pera~\cite[Section~2]{MR1646607} consider \(T\)-periodic
Carath\'eodory tangent vector fields on a manifold embedded in
Euclidean space and study periodic solutions and their global
branches, but do not state the above continuation principle
for a general Cauchy problem. J\'o\'zwikowski and Respondek
\cite[Definition~11, Theorem~7, and Lemma~1]{MR3533505}
obtain unique local flows under the substantially stronger assumption
of \(C^1\)-regularity with respect to the state variable, while the
classical escape lemma for smooth autonomous vector fields can be
found in~\cite[Lemma~9.19]{le1}. Since these references do not state the
continuation result for measurable-in-time, possibly nonunique solutions in
the precise form needed here, we include a proof.
\end{remark}

\begin{proof}[Proof of Theorem \ref{cyfx1}]
We first prove assertion \textup{(i)}. The Euclidean Carath\'eo\-dory existence
theorem \cite[Theorem~5.1, p.~28]{MR587488} requires continuity with respect
to the state variable for every fixed time, whereas assumption \textup{(ii)}
provides this continuity only for a.e. time. Let \(N_0\subset\mathbb R\) be a
measurable null set outside which assumption \textup{(ii)} holds, and define
\[
 \overline b(t,x):=
 \begin{cases}
  b(t,x),&t\notin N_0,\\
  O_x,&t\in N_0.
 \end{cases}
\]
Then \(x\mapsto\overline b(t,x)\) is continuous for every \(t\), while
assumptions \textup{(i)} and \textup{(iii)} remain valid. Moreover,
\(b\) and \(\overline b\) have exactly the same solutions because the
differential equation is required to hold only a.e. It is therefore enough to
prove assertion \textup{(i)} for \(\overline b\). For notational simplicity,
throughout the remainder of the proof of assertion \textup{(i)}, the symbol
\(b\) denotes \(\overline b\).

Choose a coordinate neighborhood \(U\subset D\) of \(x_0\), with coordinate map
\(\varphi:U\to\varphi(U)\subset\mathbb R^n\), and define
\[
 \widehat b(t,z)
 :=d\varphi_{\varphi^{-1}(z)}
 b\bigl(t,\varphi^{-1}(z)\bigr),
 \qquad (t,z)\in\mathbb R\times\varphi(U).
\]
For each fixed \(z\), the map \(t\mapsto\widehat b(t,z)\) is measurable,
and, for every \(t\), the map \(z\mapsto\widehat b(t,z)\) is continuous.
For \(a\in\mathbb R^n\) and \(R>0\), write
\begin{align}\label{euclidean-coordinate-balls}
 B(a,R)&:=\{z\in\mathbb R^n:|z-a|<R\},
 &\overline B(a,R)&:=\{z\in\mathbb R^n:|z-a|\leq R\},
\end{align}
where \(|\cdot|\) is the Euclidean norm. Thus \(B(a,R)\) is the open
Euclidean ball with centre \(a\) and radius \(R\), and \(\overline B(a,R)\)
is its closure in \(\mathbb R^n\). Because \(\varphi(U)\) is open, there
exists \(R>0\) such that
\[
 \overline B(\varphi(x_0),R)\subset\varphi(U).
\]
Set
\[
 V:=\varphi^{-1}\bigl(B(\varphi(x_0),R)\bigr).
\]
Since \(\varphi\) is a homeomorphism and \(\overline B(\varphi(x_0),R)\) is
compact, \eqref{euclidean-coordinate-balls} gives
\[
 \overline V
 =\varphi^{-1}\bigl(\overline B(\varphi(x_0),R)\bigr)
 \subset U,
\]
and \(\overline V\) is compact. Moreover,
\(\varphi(V)=B(\varphi(x_0),R)\) is convex. Since \(d\varphi\) is continuous and
\(\overline V\) is compact, there exists \(C>0\) such that
\[
 |d\varphi_x(X)|\le C|X|,
 \qquad x\in\overline V,\quad X\in T_xM.
\]
For every compact interval \(L\subset\mathbb R\), assumption \textup{(iii)},
applied to the compact set \(L\times\overline V\), gives a function
\(I_{L\times\overline V}\in L^1(L)\) such that
\[
 |\widehat b(t,z)|
 \le C I_{L\times\overline V}(t)
\]
for a.e. \(t\in L\) and all \(z\in\varphi(\overline V)\). Thus
\(\widehat b\) satisfies the Carath\'eodory conditions locally near
\((t_0,\varphi(x_0))\). By this theorem, there exist
\(\delta>0\) and an absolutely continuous local solution \(z(\cdot)\)
through \(\varphi(x_0)\). By decreasing \(\delta\) if necessary, its
image is contained in \(\varphi(V)\), and it satisfies
\[
 \dot z(t)=\widehat b(t,z(t)),\qquad z(t_0)=\varphi(x_0),
\]
on \((t_0-\delta,t_0+\delta)\). Since
\(d(\varphi^{-1})\) is bounded on the compact set
\(\overline{\varphi(V)}\), there exists \(C'>0\) such that
\[
 \bigl|d(\varphi^{-1})_w(v)\bigr|\leq C'|v|,
 \qquad w\in\overline{\varphi(V)},\quad v\in\mathbb R^n.
\]
For \(z_1,z_2\in\varphi(V)\), convexity ensures that the line segment
\(\eta(r)=(1-r)z_1+rz_2\), \(0\leq r\leq1\), is contained in
\(\varphi(V)\). By the definition of \(\rho\),
\eqref{riemannian-curve-length}, and Lemma \ref{cr},
\[
 \rho\bigl(\varphi^{-1}(z_1),\varphi^{-1}(z_2)\bigr)
 \leq L_g(\varphi^{-1}\circ\eta)
 \leq C'|z_1-z_2|.
\]
Thus \(\varphi^{-1}\) is Lipschitz from \(\varphi(V)\), endowed with the
Euclidean distance, to \(M\), endowed with \(\rho\). Define
\[
 x(\cdot):=\varphi^{-1}\circ z(\cdot).
\]
Since \(z(\cdot)\) is absolutely continuous, for every compact subinterval
\([a,b]\subset(t_0-\delta,t_0+\delta)\) and \(a\leq s\leq t\leq b\),
\[
 \rho(x(s),x(t))
 \leq C'|z(s)-z(t)|
 \leq C'\int_s^t|\dot z(r)|\,dr.
\]
It follows from Definition \ref{jDlx} that \(x(\cdot)\) is locally absolutely
continuous. At a.e. time at which \(z(\cdot)\) is differentiable,
\eqref{pointwise-velocity}, the definition of \(\widehat b\), and Lemma
\ref{cr} give
\[
 \dot x(t)
 =d(\varphi^{-1})_{z(t)}\dot z(t)
 =d(\varphi^{-1})_{z(t)}\widehat b(t,z(t))
 =d(\varphi^{-1})_{z(t)}d\varphi_{x(t)}b(t,x(t))
 =b(t,x(t))
\]
for a.e. \(t\in(t_0-\delta,t_0+\delta)\). Therefore, \(x(\cdot)\) is a
solution of \eqref{dxb} on this interval.

We now return to the original vector field \(b\) and prove assertion
\textup{(ii)}. Let \(x:[t_0,c_2)\to D\) be a solution of \eqref{dxb}. We first
suppose that \(c_2<+\infty\) and that \(x(\cdot)\) cannot be extended beyond
\(c_2\), and prove the statement for the right endpoint.

Assume, by contradiction, that the conclusion fails.
Then there exists a compact set \(K\subset D\) such that, for every
\(\tau\in[t_0,c_2)\), one can find \(t_\tau\in(\tau,c_2)\) satisfying
\(x(t_\tau)\in K\). In particular, for every sufficiently large
\(j\in\mathbb N\), there exists
\[
    t_j\in \left(c_2-\frac1j,c_2\right)
\]
such that
\begin{equation}\label{xtjk}
    x(t_j)\in K .
\end{equation}
We divide the argument into four steps.

\medskip

\noindent\textit{Step 1.}\,We construct a compact neighborhood of \(K\)
contained in \(D\).
For a subset \(A\subset M\), define
\begin{equation}\label{day}
    d_A(y):=\inf_{z\in A}\rho(y,z),
    \qquad y\in M.
\end{equation}

If \(D=M\), choose any \(r>0\). If \(D\neq M\), then
\(M\setminus D\neq\emptyset\). Since \(D\) is open, \(M\setminus D\) is
closed. For \(y\in M\), define
\[
    d_{M\setminus D}(y)
    :=
    \inf_{\eta\in M\setminus D}\rho(y,\eta).
\]
By \cite[Lemma~6.29, p.~175]{lee8}, the function
\(d_{M\setminus D}\) is continuous on \(M\). Moreover,
\[
    y\in D
    \quad\Longleftrightarrow\quad
    d_{M\setminus D}(y)>0 .
\]
Since \(K\subset D\) is compact, we have
\[
    \delta
    :=
    \inf_{y\in K} d_{M\setminus D}(y)
    >0 .
\]
In this case, choose \(r\in(0,\delta)\). In either case, set
\begin{equation}\label{dkr}
    K_r:=\{y\in M: d_K(y)\le r\}.
\end{equation}
We claim that
\begin{equation}\label{krd}
    K_r\subset D .
\end{equation}
This is immediate if \(D=M\). Suppose that \(D\neq M\), and let
\(z\in K_r\). Since \(K\) is compact, there exists
\(z_0\in K\) such that
\[
    d_K(z)=\rho(z,z_0)\le r .
\]
For any \(\eta\in M\setminus D\), the triangle inequality for \(\rho\)
gives
\[
    \rho(z,\eta)
    \ge \rho(z_0,\eta)-\rho(z,z_0).
\]
Taking the infimum over \(\eta\in M\setminus D\), we obtain
\[
    d_{M\setminus D}(z)
    \ge d_{M\setminus D}(z_0)-\rho(z,z_0)
    \ge \delta-r>0 .
\]
Thus \(z\in D\), and therefore \eqref{krd} follows.

\medskip

\noindent\textit{Step 2.}
Since \(M\) is complete, the Hopf--Rinow theorem
\cite[Theorem~2.8, p.~146]{c} implies that closed and bounded subsets of
\(M\) are compact. Since \(d_K\) is continuous, \(K_r=\{y\in M:d_K(y)\le r\}\) is closed.
Moreover, because \(K\) is compact, \(K_r\) is bounded. Hence, by the
Hopf--Rinow theorem, \(K_r\) is compact.

Consequently,
\[
    J_r:=[t_0,c_2]\times K_r
\]
is a compact subset of \(\mathbb R\times D\). By the assumed local
integrable boundedness of \(b\), there exists \(I_{J_r}\in L^1(t_0,c_2)\) such
that
\begin{align}\label{bijt}
    |b(t,y)|\le I_{J_r}(t),\quad a.e.\, t\in[t_0,c_2],\quad \forall\,y\in K_r.
\end{align}
Since \(I_{J_r}\in L^1(t_0,c_2)\), we may choose \(j\) sufficiently large such
that \(t_j\in(t_0,c_2)\) and
\begin{equation}\label{smalltail}
    \int_{t_j}^{c_2} I_{J_r}(s)\,ds<r .
\end{equation}

\medskip

\noindent\textit{Step 3.}
We claim that
\begin{equation}\label{staykr}
    x(t)\in K_r,
    \qquad t_j\le t<c_2 .
\end{equation}
Suppose otherwise. Since \(x(t_j)\in K\) by \eqref{xtjk}, we have
\(d_K(x(t_j))=0\), and the set $\{t\in(t_j,c_2): d_K(x(t))>r\}$ is nonempty. Thus, we set
\[
    \sigma
    :=
    \inf\{t\in(t_j,c_2): d_K(x(t))>r\}
\]
By continuity of \(x(\cdot)\) and \(d_K(\cdot)\), we have
\(\sigma\in(t_j,c_2)\) and
\[
    d_K(x(t))\le r,\qquad t\in[t_j,\sigma],
    \qquad
    d_K(x(\sigma))=r .
\]
In particular, \(x(t)\in K_r\) for all \(t\in[t_j,\sigma]\). Hence, using
the definition of \(d_K(\cdot)\), \eqref{bijt}, and Lemma
\ref{lem:distance-ac-curve}, we obtain
\[
\begin{aligned}
    r
    =d_K(x(\sigma))
    &\le \rho(x(t_j),x(\sigma))                                      \\
    &\le \int_{t_j}^{\sigma} |\dot x(s)|\,ds                         \\
    &= \int_{t_j}^{\sigma} |b(s,x(s))|\,ds                            \\
    &\le \int_{t_j}^{\sigma} I_{J_r}(s)\,ds                                \\
    &< r,
\end{aligned}
\]
where the last inequality follows from \eqref{smalltail}. This is a
contradiction. Therefore \eqref{staykr} holds.

Set
\[
 K^*:=x([t_0,t_j])\cup K_r,
 \qquad J:=[t_0,c_2]\times K^*.
\]
Since \(x(\cdot)\) is continuous, \(x([t_0,t_j])\) is a compact subset of \(D\).
Thus \(K^*\) is compact, and so is \(J\). Assumption \textup{(iii)} gives a
function \(I_J\in L^1(t_0,c_2)\) such that
\begin{equation}\label{bijt-full-trajectory}
 |b(t,y)|\le I_J(t),
 \quad\text{for a.e. }t\in[t_0,c_2]\text{ and every }y\in K^*.
\end{equation}
By \eqref{staykr}, \(x(t)\in K^*\) for every \(t\in[t_0,c_2)\).

\medskip

\noindent\textit{Step 4. Extension beyond \(c_2\).}
Let \(s,t\in[t_j,c_2)\) with \(s<t\). Since \(x(\ell)\in K_r\) for
\(\ell\in[t_j,c_2)\), Lemma \ref{lem:distance-ac-curve}, the differential
equation, and \eqref{bijt-full-trajectory} give
\[
\begin{aligned}
    \rho(x(s),x(t))
    \le \int_s^t |\dot x(\ell)|\,d\ell
    = \int_s^t |b(\ell,x(\ell))|\,d\ell\le \int_s^t I_J(\ell)\,d\ell .
\end{aligned}
\]
Since \(I_J\in L^1(t_0,c_2)\), for every \(\varepsilon>0\) there exists
\(\delta>0\) such that
$
    \int_s^t I_J(\ell)\,d\ell<\varepsilon$,
whenever \(c_2-\delta<s\le t<c_2\). Hence
\[
    \rho(x(s),x(t))
    \leq
    \int_s^t I_J(\ell)\,d\ell
    <\varepsilon
\]
for all \(c_2-\delta<s\le t<c_2\). Thus \(x(\cdot)\) is Cauchy as
\(t\to c_2^{-}\). Since \(x([t_j,c_2))\subset K_r\) and \(K_r\) is compact, there exists
\(x_{c_2}\in K_r\) such that
\[
    \lim_{t\to c_2^-}x(t)=x_{c_2}.
\]

Applying the local existence result in assertion \textup{(i)} with initial
condition \(y(c_2)=x_{c_2}\), we obtain \(\delta_{c_2}>0\) and a solution
\(y:(c_2-\delta_{c_2},c_2+\delta_{c_2})\to D\) of
\[
\begin{cases}
    \dot y(t)=b(t,y(t)),
    & \text{for a.e. }t\in(c_2-\delta_{c_2},c_2+\delta_{c_2}),\\
    y(c_2)=x_{c_2}.
\end{cases}
\]
Define
\[
    \widetilde x(t)
    :=
    \begin{cases}
        x(t), & t_0\leq t<c_2,\\
        y(t), & c_2\le t<c_2+\delta_{c_2}.
    \end{cases}
\]

It remains to show that the pasted curve \(\widetilde x(\cdot)\) is locally
absolutely continuous on \([t_0,c_2+\delta_{c_2})\). Let
\([a,b]\subset [t_0,c_2+\delta_{c_2})\) be arbitrary. If \(b<c_2\), then
\(\widetilde x(\cdot)=x(\cdot)\) on \([a,b]\), and hence
\(\widetilde x(\cdot)\) is absolutely continuous on \([a,b]\). If
\(a\ge c_2\), then \(\widetilde x(\cdot)=y(\cdot)\) on
\([a,b]\), and again \(\widetilde x(\cdot)\) is absolutely continuous on
\([a,b]\), because \([a,b]\) is a compact subinterval of the interval of
existence of \(y(\cdot)\).

If \(b=c_2\), choose \(b'\in(c_2,c_2+\delta_{c_2})\). Once the result is
proved on \([a,b']\), it follows on \([a,b]\) by restriction. It therefore
remains to consider the case \(a<c_2<b\). Since \(x(t)\to x_{c_2}\) as
\(t\to c_2^{-}\), we extend \(x(\cdot)\) to \([a,c_2]\) by setting
\[
    x(c_2):=x_{c_2}.
\]
Recall that \(I_J\in L^1(t_0,c_2)\). Since \(a\geq t_0\), we have
\(I_J\in L^1(a,c_2)\). By
\eqref{bijt-full-trajectory} and the differential equation, we have
\(|\dot x(t)|\le I_J(t)\) for a.e. \(t\in[a,c_2)\).
For \(a\le s\le t<c_2\) with \(s>t_0\), Lemma
\ref{lem:distance-ac-curve} and the preceding bound give
\[
    \rho(x(s),x(t))
    \le
    \int_s^t I_J(\tau)\,d\tau .
\]
If \(s=t_0\), the same inequality follows by letting \(s\to t_0^+\), using
the continuity of \(x(\cdot)\) and the integrability of \(I_J\).
Letting \(t\to c_2^{-}\), we obtain
\[
    \rho(x(s),x_{c_2})
    \le
    \int_s^{c_2} I_J(\tau)\,d\tau,
    \qquad a\le s<c_2.
\]
Therefore, by Definition \ref{jDlx}, the extended curve \(x(\cdot)\) is
absolutely continuous on \([a,c_2]\). On the other hand, assertion
\textup{(i)} gives a locally absolutely continuous solution \(y(\cdot)\) on
\((c_2-\delta_{c_2},c_2+\delta_{c_2})\). Since \([c_2,b]\) is a compact
subinterval of this interval, the restriction of \(y(\cdot)\) to \([c_2,b]\) is
absolutely continuous. Hence there exists a function
\(m_{y,c_2,b}\in L^1(c_2,b)\) such that
\[
    \rho(y(s),y(t))
    \le
    \int_s^t m_{y,c_2,b}(\tau)\,d\tau,
    \qquad c_2\le s\le t\le b.
\]
Define
\[
    m(\tau)
    =
    \begin{cases}
        I_J(\tau), & a\le \tau<c_2,\\
        0, & \tau=c_2,\\
        m_{y,c_2,b}(\tau), & c_2<\tau\le b.
    \end{cases}
\]
Then \(m\in L^1(a,b)\). We claim that
\[
    \rho(\widetilde x(s),\widetilde x(t))
    \le
    \int_s^t m(\tau)\,d\tau,
    \qquad a\le s\le t\le b.
\]
Indeed, if \(s,t\) lie on the same side of \(c_2\), this follows from the
absolute continuity of the corresponding piece. If \(s<c_2<t\), then,
using \(x_{c_2}=y(c_2)\), we have
\[
\begin{aligned}
    \rho(\widetilde x(s),\widetilde x(t))
    &=
    \rho(x(s),y(t))                                      \\
    &\le
    \rho(x(s),x_{c_2})+\rho(y(c_2),y(t))                  \\
    &\le
    \int_s^{c_2} I_J(\tau)\,d\tau
    +
    \int_{c_2}^{t} m_{y,c_2,b}(\tau)\,d\tau              \\
    &=
    \int_s^t m(\tau)\,d\tau .
\end{aligned}
\]
Thus \(\widetilde x(\cdot)\) is absolutely continuous on \([a,b]\).

Since the compact interval \([a,b]\subset[t_0,c_2+\delta_{c_2})\) was
arbitrary, \(\widetilde x(\cdot)\) is locally absolutely continuous on
\([t_0,c_2+\delta_{c_2})\).
Moreover, \(\widetilde x(t_0)=x_0\), and
\[
 \dot{\widetilde x}(t)=b(t,\widetilde x(t))
\]
for a.e. \(t\in(t_0,c_2+\delta_{c_2})\), because this equality holds on
each side of \(c_2\) and the single time \(c_2\) is negligible. Thus
\(\widetilde x(\cdot)\) is a continuation of \(x(\cdot)\) beyond \(c_2\), which
contradicts the assumption that \(x(\cdot)\) cannot be extended beyond
\(c_2\). Therefore, \(x(\cdot)\) must leave every compact subset of \(D\) as
\(t\to c_2^-\).

We finally prove the statement for the left endpoint. Let
\(x:(c_1,t_0]\to D\) be a solution of \eqref{dxb}. Suppose that
\(c_1>-\infty\) and that \(x(\cdot)\) cannot be extended beyond \(c_1\) to the
left. Define
$
 z(s):=x(-s)$ for $ s\in[-t_0,-c_1)$,
and
$
 \widetilde b(s,q):=-b(-s,q)$ for $ (s,q)\in\mathbb R\times D$.
The change of variables \(t=-s\) shows that \(\widetilde b\) satisfies the same
Carath\'eodory-type assumptions as \(b\). In particular, for a compact subset
of \(\mathbb R\times D\), condition \textup{(iii)} for \(\widetilde b\) follows
by applying the corresponding condition for \(b\) to the compact set obtained
by reflecting its time component; local integrability is preserved under this
reflection. Moreover, \(z:[-t_0,-c_1)\to D\) is locally absolutely continuous,
\(z(-t_0)=x_0\), and
\[
 \dot z(s)=-\dot x(-s)=-b(-s,z(s))=\widetilde b(s,z(s))
\]
for a.e. \(s\in(-t_0,-c_1)\). Thus, \(z(\cdot)\) is a solution of the reflected
system. If it could be extended beyond its right endpoint \(-c_1\), reversing
time would give an extension of \(x(\cdot)\) beyond \(c_1\) to the left. Hence,
the right-endpoint argument applied to \(z(\cdot)\) shows that \(z(\cdot)\)
leaves every compact subset of \(D\) as \(s\to(-c_1)^-\). Equivalently,
\(x(\cdot)\) leaves every compact subset of \(D\) as \(t\to c_1^+\).
The proof is complete.
\end{proof}
 
  \bibliographystyle{amsplain_no_bysame}
\bibliography{ISS_refs}
\providecommand{\bysame}{\leavevmode\hbox to3em{\hrulefill}\thinspace}
\providecommand{\MR}{\relax\ifhmode\unskip\space\fi MR }
\providecommand{\MRhref}[2]{%
  \href{http://www.ams.org/mathscinet-getitem?mr=#1}{#2}
}
\providecommand{\href}[2]{#2}

\end{document}